\documentclass[10pt]{article}
\usepackage{cite}
\usepackage{amsmath}
\usepackage{amsfonts}
\usepackage{amssymb}
\usepackage{fontenc}
\usepackage{graphicx}
\usepackage{mathrsfs}
\usepackage{yfonts}
\usepackage{hyperref}
\usepackage{enumerate}
\usepackage{mathtools}
\usepackage{textcomp}
\usepackage{lmodern, lipsum}
\usepackage{calligra}
\DeclareMathAlphabet{\mathcalligra}{T1}{calligra}{m}{n}
\usepackage{dsfont}
\usepackage{color}

\usepackage{todonotes}
\usepackage{ucs}
\usepackage{graphicx}
\usepackage{mathrsfs}
\usepackage{yfonts}
\usepackage{hyperref}
\usepackage{bm}

\usepackage{enumerate}
\usepackage{mathtools}
\usepackage{textcomp}
\usepackage{lmodern, lipsum}
\usepackage{calligra}
\DeclareMathAlphabet{\mathcalligra}{T1}{calligra}{m}{n}
\usepackage{dsfont}
\usepackage{color}
\usepackage{todonotes}

\newtheorem{thm}{Theorem}
\newtheorem{thm*}{Theorem*}

 \newtheorem{claim}[thm]{Claim}

 \newtheorem{lem}[thm]{Lemma}

\newtheorem{prop}[thm]{Proposition}
\newtheorem{rem}[thm]{Remark}

\newenvironment{proof}[1][Proof]{\textbf{#1.} }{\ \rule{0.5em}{0.5em}}

\newtheorem{cor}{Corollary}[section]
\newtheorem{defn}{Definition}[section]
\newtheorem{exmp}{Example}[section]

\newcommand{\p}{\partial}
\newcommand{\pd}[1]{\partial_{#1}}

\newcommand{\R}{\mathbb{R}}
\newcommand{\abs}[1]{\left|#1\right|}
\newcommand{\f}{\phi}

\newcommand{\vf}{\varphi}

\renewcommand{\r}{\rho}

\newcommand{\gam}{\gamma}
\newcommand{\lam}{\lambda}
\renewcommand{\a}{\alpha}
\renewcommand{\b}{\beta}

\newcommand{\iso}{\simeq}

\newcommand{\Z}{\mathbb{Z}}
\newcommand{\set}[1]{\left\{ #1 \right\}}

\renewcommand{\L}{\mathcal{L}}

\newcommand{\ep}{\epsilon}
\newcommand{\inp}[1]{\langle #1 \rangle}

\newcommand{\til}[1]{\tilde{#1}}
\newcommand{\im}{\operatorname{Im}}

\newcommand{\grad}{\nabla}
\newcommand{\del}{\delta}
\newcommand{\id}{\operatorname{Id}}

\newcommand{\wed}{\wedge}

\newcommand{\s}{\sigma}
\newcommand{\sig}{\sigma}
\renewcommand{\S}{\Sigma}
\newcommand{\Sig}{\Sigma}

\newcommand{\norm}[1]{\|#1\|}

\renewcommand{\t}{\tau}

\newcommand{\infmax}{\operatorname{infmax}}

\newcommand{\imply}{\Rightarrow}
\newcommand{\F}{\varPhi}
\newcommand{\N}{\mathbb{N}}

\newcommand{\Lam}{\Lambda}

\newcommand{\bs}{\backslash}
\newcommand{\codim}{\operatorname{codim}}
\newcommand{\comp}{\circ}

\newcommand{\Lip}[1]{\norm{#1}_{\operatorname{Lip}}}
\renewcommand{\th}{\theta}

\newcommand{\at}[1]{\big |_{#1}}
\newcommand{\jet}[1]{J^{1}\R^{#1}}
\newcommand{\m}{\mu}

\newcommand{\opname}[1]{\operatorname{#1}}
\newcommand{\z}{\zeta}

\newcommand{\vfd}{\phi}

\newcommand{\mF}{\mathcal{F}}
\newcommand{\mP}{\mathcal{P}}
\newcommand{\x}{\xi}

 \newcommand{\gfqi}{\texttt{g.f.q.i.}}
 \newcommand{\fG}{\textfrak{G}}
 \newcommand{\br}{\bar{R}}
 \newcommand{\bx}{\bar{x}}
 \newcommand{\by}{\bar{y}}
 \newcommand{\bz}{\bar{z}}
 \newcommand{\xl}{x_{\lam}}
 \newcommand{\co}[1]{\operatorname{co}\set{#1}}
 
 \newcommand{\pdehs}{\Theta}
\newcommand{\evat}[2]{\left. #1 \right|_{#2}}
\newcommand{\cet}[2]{\set{\left. #1 \right| #2}}
\newcommand{\conmani}{\Xi}

\title%[Geometric and viscosity solutions for Cauchy problem]
{Geometric and viscosity solutions 
for the Cauchy problem of first order}  
\author%[J. Castillo Colmenares]
{Juliho David Castillo Colmenares}
\begin{document}
\maketitle

 \begin{abstract}
There are two kinds of solutions of the Cauchy problem of first order, the viscosity 
solution and the more geometric minimax solution and in general they are different. 
The aim of this thesis is to show how they are related: 
iterating the minimax procedure during shorter and shorter time intervals 
one approaches the viscosity solution. This can be considered as an extension 
to the contact framework of the result of Q. Wei \cite{W1} in the symplectic case.
 \end{abstract}

%\pagenumbering{arabic}

\tableofcontents
%\clearpage

%\input{./jdcastillo_thesis_01.tex}
%\part{Preeliminares}

\section{Foreword}
This article is concerned with the study of weak solutions to the following Cauchy problem for
the Hamilton-Jacobi equation:
%Considere el problema de Cauchy
\begin{equation}
 \label{HJ}
 \tag{HJ}
  \begin{cases}
   \p_tu(t,x)+H(t,x,\p_x u(t,x),u(t,x))=0, & t\in(0,T]\\
   u(0,t)=v(x), & x\in\R^k.
  \end{cases}
 \end{equation}
In particular, wet focus on the relation between the notion of viscosity solution (introduced by Crandall, Evans and Lions) and the more geometric one of minmax solution (due to Chaperon and Sikorav); this latter is based on a minmax procedure
and relies on the existence of suitable generating families. It is known that when the Hamiltonian
$H(t, x, y, z)$ is not assumed to be convex in the momentum component y, then these two notions
do not (in general) lead to the same set of solutions.  For example, see \cite[prop. 3.8]{wei2013front}.

The main aim of this paper is to describe an iterated minmax procedure - performed on finer
and finer partitions of the time interval - that, in the limit, allows one to recover viscosity solutions
from the minmax ones.

A result in the very same spirit has been recently obtained by Qiaoling Wei \cite{W1}, in the case of Hamiltonians that depend only on $(t, x, p)$, the so-called symplectic case. This work can be seen as a (technical) extension of Wei's result and techniques to
the contact case, namely for Hamiltonians depending on $(t, x, y, z) \in [0, T ] \times \R^k \times \R^k \times \R.$

Although the article  is a technical adaptation - yet, not immediate - of Wei's approach to this
context, we believe that it might be useful as a possible reference for further applications and developments.

 The classical method (section \ref{subsec:ec_car}) to solve this
 %El m\'etodo cl\'asico (secci\'on \ref{subsec:ec_car}) para resolver este
 problem, for $v\in C^2$ and a short time interval, consists in solving 
 %problema, para $v\in C^{2}$ y un tiempo corto de tiempo, consiste en resolver
characteristics equations
%las ecuaciones caracter\'isticas
 \begin{align}
 \label{bh:1.3a} 
 \tag{H1}
 \dot{x}&=\pd{y}H  \\
 \label{bh:1.3b}
 \tag{H2}
 \dot{y}&=-\pd{x}H-y\pd{z}H \\
 \label{bh:1.3c}
 \tag{H3}
 \dot{z}&=y\pd{y}H -H,
 \end{align}
to get the characteristic lines, that is, the trajectories 
$$
\F(t, q_{0})\dot{=}(t, x(t,q_{0}), y(t,q_{0}), z(t, q_{0}) );
$$ 
where $x(\cdot),y(\cdot),z(\cdot)$ are the solutions to characteristic equations with initial conditions $q_{0}=( x_0, Dv(x_0), v(x_0))$ at time $t=0$ and then obtain the solution $u(t, x)$ of the Cauchy problem as follows: 
setting $u_t (x) \dot{=} u(t, x),$ for fixed $t,$ and its 1-jet  $j^1u_t(x)=(x,du_t(x),u_t(x))$, 
the image of $j^1u_t$ is the section at time $t$ of c 
$$\L\left( \Lam_{0} \right)=\bigcup_{t\in[0,T]}\F(t, \Lam_{0})=\cet{\F(t)}{t\in[0,T], \F(0)\in \Lam_0 }$$ where $\Lam_{0}=\set{(0,j^1v(x))}$.
% while $d_tu$ is given by the equation. 

This procedure does not  yield a global solution of the problem in the 
whole interval $[0,T]$, as the geometric solution $\L$ it is not always the set
$\{(t,j^1u_t(x))\}$ for a function $u(t, x)$: 
According to \cite[section 3.2]{arnold2013singularities} the \emph{wavefront} 
$$\mF=\pi\left( \L(\Lam_{0}) \right),$$ 
(where $\pi\left( t,x,y,z \right)=(t,x,z)$)
obtained in $(t,x,y,z)$ space by solving the equation
$$dz = -H (t, x, y, z) dt + y dx$$ restricted to $(t,x,y,z)\in \L$ it is
not the graph of a function: 
the projected characteristics $$\pi\left( \F(t, q_{0}) \right)=(t,x(t,q_{0}), z(t,q_{0})), \; q_{0}\in\Lam_{0}$$ may cross after some time. For example, characteristics for conservation laws \cite[example 3.2.5]{evans1997partial}

Whereas in some applications, e.g. to geometrical optics, the wavefront $\mF$ 
can be considered as a solution of the physical problem, 
one is interested in a single-valued solution $u(t,x)$.  
Assuming that the projection of $\mF$ into $(t, x)$ space is onto, one can
construct such a solution as a section of the wavefront, 
selecting a single $u$ over each $(t, x).$ 
When the function $H$ is sufficiently convex with respect to $y$ 
(and $v$ is not too wild at infinity), such a ``selector'' consists in 
choosing for $u(t, x)$ the smallest $u$ with $(t,x,u) \in\mF.$

This min solution happens to be the ``viscosity solution''  which was
first introduced as the viscosity limit when $\ep \to 0^{+}$ of the
solution of the Cauchy problem for the viscous equation 
$$\p_tu(t,x)+H(t,x,\p_{x}u, u)=\ep\triangle_{x}u(t,x),$$ 
and afterwards got a general definition for general nonlinear first
order partial differential equations in the work of Crandall, Evans
and Lions \cite{CEL84} \cite{CEL84,BCD}.

In the non-convex case the viscosity solution may not 
be a section of the wavefront (see for example \cite{Che}). On the other
hand, Chaperon introduced in \cite{Ch01} weak solutions whose graph is
a section of the wavefront, obtained by a ``minimax'' procedure which
generalizes the minimum considered in the convex case and relies on
the existence of suitable generating families for the geometric solution. 

Let's explain in more detail this procedure: First, considere the \emph{legendrian submanifold} $\Lam=\vf^{t}\left( j^{1}v \right) \subset J^{1}\R^{k},$ for some $\R^{k},$ where $\vf^{t}$ is the flow generated by \eqref{bh:1.3a}-\eqref{bh:1.3c} and therefore $\F(t)=\left( t , \vf(t) \right).$ It turns out that there exist a so called \emph{generating function quadratic at infinity} $$S:[0,T]\times \R^{k} \times \R^{q}, \left( t,x,\xi \right)\to S_{H,v}^{t}(x,\xi)$$ (for some $\R^{q},$ a family of parameters $\xi$) such that
\begin{equation}
 \label{0.1}
 \Lam=\cet{\left( x, \partial_{x}S_{H,v}^{t}(x,\xi), S_{H,v}^{t}(x,\xi) \right)}{\partial_{\xi}S_{H,v}^{t}(x,\xi)=0}.
\end{equation}

Now we can define a \emph{minimax selector} such that 
$$
u(t,x)=\infmax_{\xi} S_{H,v}^{t}(x,\xi)
$$
is a generalized solution of Cauchy problem \eqref{HJ}. It is called the \emph{minimax solution.} But in general this is not a classical solution. Indeed $u_{t}\in C^{Lip}\left( \R^{k} \right).$ Although, we have been considering $v\in C^{2}(\R^{k}),$ in the more general setting of \emph{Clarke calculus,} we can consider $v\in C^{Lip}(\R^{k}),$ and thus for a given $H\in C^{2}(J^{1}\R^{k}),$ we obtain an operator $$R_{H}^{0,t}:C^{Lip}(\R_{H}^{k})\to C^{Lip}(\R^{k}), \; v(\cdot)\mapsto \infmax_{\xi} S_{H,v}^{t}(\cdot, \xi).$$ Of course, we can take another initial initial time $t_{0}=s\neq 0,$ for a different boundary condition $u(s,x)=v(x).$

One may to try to get a solution as a limit obtained by dividing a given time 
interval into small pieces and iterating the minimax procedure step by step. 
Our goal is to show that when the size of the time intervals go to zero, 
one indeed gets the viscosity solution as the limit:

\begin{thm}[Main theorem]
\label{thm:main}
 Suppose 
 $$
 \begin{cases}
  H\in C_c^2( [0,T] \times J^1(\R^{k}) )\\
  v\in C^{Lip}(\R^{k})
 \end{cases}
$$
Then the viscosity solution is the limit of iterated minimax solutions
for problem \eqref{HJ}
on $[0,T]$.
\end{thm}

Indeed, it is the aim of this dissertation to give a complete proof of this results.  Let's explain the meaning if theorem \ref{thm:main} in more detail. Take a partition of $[0,T]:$
$$
0=t_{0}<t_{1}<...<t_{N}=T,
$$
and for $i=0,...,N-1,$ define the so called \emph{iterated minimax solution}
$$
\begin{cases}
 u_{0}(t,x)=v(x) \\
 u_{i+1}(t,x)=R_{H}^{t_{i},t_{i+1}}\left( u_{i}(t,x) \right).
\end{cases}
$$

{Our main result result establishes that $u_{N}(t,x)$ converges uniformily on compacts to the viscosity solution of Cauchy problem \eqref{HJ} when $$\sup\set{t_{i+1}-t_{i}}\to 0.$$}

This extends the result obtained by Q. Wei 
\cite{W1} in the symplectic framework to the contact one. Indeed, Prof. Wei was student of Prof. Marc Chaperon, who in \cite{Ch01} had already defined generating families and minimax in the contact setting. But in the way that minimax solutions were defined in that paper, they weren't suitable for constructing a generating family similar to Wei managed to construct.

However, in his Ph.D. thesis \cite{Bh}, Prof. Mohan Bhupal defined a generating function for contactomorphisms in a more convinient manner, and we took advantage of his work to develope a theory in contact framework which generalize the results obtained by Wei.

As Prof. Chaperon told us in a brief talk a couple years ago, this problem was important in order to understand better the geometry of viscosity solutions. Although the theorem in symplectic was already important because its relations with classical mechanics, our generalization allows to study a more general class of problems related with non-linear first order partial differential equations. 

For convenience of the reader, we give a very brief guide to this thesis: At first, in part \ref{ch:gf}, we deduce a explicit formula to obtain a generating family for legendrian submanifolds of the form $\F(t,j^{1}v)$, and generalize this construction to a more general setting of Clarke calculs in order to construct iterated minimax solutions.

Next, in part \ref{gen-sol}, we define minimax selector, iterated minimax solutions and prove the relations established above with viscosity solutions. We followed closely the methods of Prof. Wei in order to get the desired generalization, and so is in this section where our contribution appears. 

 At the end of this work, three appendices are given: The first one is about contact topology, in particular, legendrian submanifolds and conctactomorphisms; the second one is about Clarke calculus; and the last one about minimax principles. A more detailed treatment of these topics could be consulted in \cite{book:7935}, \cite{CL1} and \cite{STR}, respectively.
 
 Finally, let's us remark, as Arnold did in his lectures ``Contact Geometry and Wave Propagation'' given at the University of Oxford 
 \begin{quote}
 ``The relations between symplectic and contact geometries are similar to those between linear algebra and projective geometry. First, the two related theories are formally more or less equivalent: every theorem in symplecticgeometry may be formulated as a contact geometry theorem, and any assertionin contact geometry may be translated into the language of symplectic geometry. Next, all the calculations look algebraically simpler in the symplectic case, but geometrically things are usually better understood when translated into the language of contact geometry.''
 \end{quote} 
Following that spirit, as far as posible, we has formulated our results as their symplectic counterpart are given on \cite{W1}. One more advice from Arnold's lectures:
\begin{quote}
	Hence one is advised to calculate symplectically but to think rather in contact geometry terms.
\end{quote}

\part{Generating Families}
\label{ch:gf}

 The main goal of this part is to construct \emph{generating families for Legendrian submanifolds} $\vf^{t}(j^{1}v),$ where $\vf^{t}$ is the flow generated by \eqref{bh:1.3a}-\eqref{bh:1.3c}. As a previous step, we have to demonstrate the existence of \emph{generating functions} for contactomorphism $\vf^{t}.$
 
 These constructions are fundamental for our work, because they will allow us to generalize in a very natural way those formulas given in \cite{W1}. Following that work, we will be  able to define explicitly \emph{iterate minimax solutions} in the next section, but in the more general setting of contact topology, that is, for Hamiltonians depending on $(t,x,\partial_{x}u,u).$
 
 At the end of the section, we extend the concept of generating families using Clarke calculus \cite{CL1}, allowing us to consider initial conditions $v\in C^{Lip}.$ This step is crucial because minimax solutions are not differentiable anymore but only Lipschitz, and we will iterate solutions of this type.
 
\section{Generating Functions}
Consider $J^1(M)=T^{*}M\times \R$ endowed with the natural contact structure
$\ker \a$ given in local coordinates $(x,y,z)$ for $T^*M \times \R$  by
$\a=dz-ydx.$ We denote by $\pi: J^1(M)\to M$ the canonical projection.
Recall that a submanifold $L$ of $J^1(M)$ is called Legendrian if
$T_pL\subset\evat{\ker \a}{p}$ for any $p\in L$ and dim $L=$ dim $M$. 

Suppose $S\in C^2(M\times \R^q)$ has fiber derivative 
$\p_{\xi}S$ transversal to $0.$ Then
$$
\Lam:=\set{( x, \p_{x}S(x,\xi), S(x,\xi) )|
\p_{\xi}S(x, \xi)=0}
$$
is a Legendrian submanifold of $J^1M$ and we say that $S$ a 
\emph{generating family} of $\Lam.$  

\label{sec:gfqi}
\begin{defn}[$\gfqi$]
  \label{defn:gfqi}
 A generating family $S\in C^2(M\times\R^q)$ is 
\emph{quadratic at infinity} ($\gfqi$ for short) if there exists a 
\emph{non degenerate quadratic form} $Q$, such that for any compact 
$K\subset M$, 
$$\abs{\p_{\xi}( S(x,\xi)- Q(\xi) )}$$
is bounded on $K\times \R^q.$
\end{defn}
Consider the sub-level sets $S^a_x :=\{\xi:S(x,\xi)\le a\}$, for $a$
large enough the homotopy type of $(S^a_x,S^{-a}_x)$ does not depend on $a$ 
and coincides with the homotopy type of $(Q^a , Q^{-a})$, so
we may write it as $(S^\infty_x, S^{-\infty}_x)$. 
If the Morse index of $Q$ is $k$, then 
\[H_i(S^\infty_x, S^{-\infty}_x;\Z_2)=H_i(Q^\infty, Q^{-\infty};\Z_2 )\cong
  \begin{cases}\Z_2,& i=k,\\ 0, & i\ne k. \end{cases}\]
For more details, see \cite[section 7.2.3]{FC2015}

  \begin{defn}
The minimax function is defined as 
\[R_S(x) := \inf_{[\s]=A}\max_{\xi\in\abs{\s}} S(x,\xi),\]
where $A$ is a generator of the homology group $H_k(S^\infty_x,S^{-\infty}_x;\Z_2)$
and $\abs{\s}$ denotes the image of the relative singular homology cycle $\s$. 
  \end{defn}

  \begin{rem}
 In \cite{W2}, Q. Wei proved that the minimax values of a function quadratic nondegenerate at infinity are equal when defined in homology or cohomology with coefficients in a field. However, by an example of F. Laudenbach, this is not always true for coefficients in a ring and, even in the case of a field, the minimax-maximin depends on the field.
 
 Although this is the most common manner to describe the \emph{minimax function}, we will give a more concrete description of the minimax in section \ref{gen-sol}.
\end{rem}
  
The function $R_S$ is determined
by $\Lam$ and does not depend on the particular choice of the $\gfqi$
$S$ according to the following result
\begin{thm}\cite[prop 2.12]{Davi-1999}
 \label{chaperon:thm:theret}
The $\gfqi$ of a Legendrian submanifold contact isotopic to the zero
section is  unique up to the following operations relating $S_1$ to $S_2$.
\begin{description}
\item[Stabilitation] $S_2(x,\xi,\eta)=S_1(x,\xi)+q(\eta)$ with $q$ a 
non degenerate quadratic form.
\item[Diffeomorphism] $S_2(x,\xi)=S_1(x,\psi(x,\xi))$ with $\psi(x,\cdot)$
a diffeomorphism $\forall x\in M$.
\end{description}
 \end{thm}
The following definition is common in the literature
 \begin{defn}[strict $\gfqi$]
  \label{defn:gfqi2}
A generating function is \emph{strictly quadratic at infinity} 
if there is a non degenerate quadratic form $Q$ such that 
$S(x,\xi)=Q(\xi)$ for $(x,\xi)$ outside some compact set (we will say strict $\gfqi$ for short).
 \end{defn}

 This definition is more appropriate to work with hamiltonians $$H\in C^1([0,T] \times J^1M),$$ for $M$ a compact manifold. However under some restrictions, both definitions of $\gfqi$ are equivalent
 \begin{prop}\cite[prop. 1.6]{Vit2006}

  Suppose there is a constante $C$ suh that
  \begin{enumerate}[(a)]
   \item $\norm{\grad( S-Q )}_{C^0}< C$
   \item $\sup\set{\abs{S-Q}|x\in \R^k, \abs{\xi}\leq r}<Cr$   
  \end{enumerate}
  Then $L_{S}$ has a strict $\gfqi$.
 \end{prop}

Recall that a diffeomorphism $\vf: J^1M\to J^1M$ is called
\emph{contactomorphism} if $D\vf(\ker \a)=\ker \a$
or equivalently $\vf^{*}\a=g \a$ with $g\in C^1(J^1M,\R-\set{0})$. 

From here, we follow the construction given in \cite[section 6]{Bhupal-2001}.

\begin{defn}\label{fg} 
Let $\vf:\jet{k} \to \jet{k}$ be a contactomorphism. A \emph{generating function} for
$\vf$ is a function $\F:\R^{2k+1}\to\R$ such that $1-\pd{z}\F(x,Y,z)$
never vanishes and the set of equalities
 \begin{align}
  \label{cx} \tag{cx}
   X-x&=\pd{Y}\F(x,Y,z)\\
   \label{cy}\tag{cy}
   Y-y&=-\pd{x}\F(x,Y,z) - y\pd{z}\F(x,Y,z)\\
   \label{cz}\tag{cz}
   Z-z&=( X-x )Y-\F(x,Y,z)
 \end{align}
is equivalent to $\vf(x,y,z)=( X,Y,Z ).$
\end{defn}
\begin{rem}
 The contactomorphism $\vf$ has compact support if and only if $\F$ does.
\end{rem}
\begin{prop}\cite[prop. 6.1]{Bhupal-2001}\label{exgf}
\begin{enumerate}[(i)]
 \item A contactomorphism $\vf:\jet{k}\to \jet{k}$ with
$\norm{\mathds1-d\vf(p)}<\dfrac{1}{2}$ for all $p\in \jet{k}$
has a unique generating function. 
  \item  If $\F\in C_c^{\infty}(\R^{2k+1})$ has  sufficiently small first and second
   derivatives, there exists a unique
   contactomorphism $\vf: \jet{k} \to \jet{k}$ having $\F$ as
   generating function
\end{enumerate}
\end{prop}

\section{Method of Characteristics}
\label{subsec:ec_car}
For $H\in C^2([0,T]\times J^1\R^k)$ let $X_{H}$ be the
associated time-dependent contact vector field given by
(H1)-(H3), and  $ t\mapsto\vf^t(q)$
be the integral curve with $\vf^0(q)=q$. 
One calls $\vf^t$ the {\em contact isotopy}  defined by $H$. We define
$\vf^{s,t}=\vf^t\circ(\vf^s)^{-1}$.
Suppose that $H$ has compact support cotained in the set $$\set{(x,y,z)\in\jet{k}:|y|\le a}$$
and let $c_H=\sup\set{|DH_t(x,y,z)|, |D^2H_t(x,y,z)|}$, then  
$\max\limits_t\Lip{X_{H}}\le (2+a)c_H$. 
\begin{lem}[Cf. \cite{W1}, lemmata 2.5,2.6]
 \label{gen=act}
 If $\del_H=\log 2/\left( (2+a)c_H \right)$, for $0<t-s<\del_H$ there is a
generating function $\F^{s,t}: \jet{k} \to \R$ for $\vf^{s,t}$. Let
$q=(x,y,z)$, $r=(X,Y,Z)$ and for $s\le\t\le t$, define 
$\vf^{s,\t}(q)=(x(\t,q),y(\t,q),z(\t,q))$, $\vf^{t,\tau}(r)=( \bx(\t,r),\by(\t,r),\bz(\t,r))$. Then 
\begin{align}  \label{eq:1}
  \F^{s,t}(x,y(t, q), z)&=\int_s^t (\dot{x}(\t, q)(y(t, q)-y(\t, q))+ H(\t,\vf^{s,\t}(q)) )d\t,\\
\label{Ft=H}
 \p_t \F^{s,t}(x,y(t, q), z)   &= H(t,\vf^{s,t}(q) ),\\
\label{Fs=-H}
\p_s\F^{s,t}(\bx(s,r),Y, \bz(s,r))  &= H(t, \vf^{t,s}(r))(\p_{z}\F^{s,t}(\bx(s,r),Y, \bz(s,r))-1 )
\end{align}
\end{lem}
\begin{proof}
From the general theory of differential equations, 
$\norm{\mathds1-d\vf^{s,t}(p)}<1$ for  $0<t-s<\del_H$ and $p\in\jet{k}$.
By Proposition \ref{exgf}, $\vf^{s,t}$ has a generating
function  $\F^{s,t}$ so that
\begin{equation}
\begin{split}
 x(t,q)-x&=\p_{y}\F^{s,t}(x,y(t,q), z)\\
 y(t,q)-y&=-\p_{x}\F^{s,t}(x,y(t,q), z)-y\p_{z}\F^{s,t}(x,y(t,q), z)\\
 z(t,q)-z&=y(t,q)(x(t,q)-x)-\F^{s,t}(x,y(t,q), z).
 \end{split}
\end{equation}
Then
\begin{align*}
 \F^{s,t}(x,y(t, q), z) &= (x(t, q)-x)y(t,q)-( z(t, q)-z )\\
  &= \int_s^t \dot{x}(\t, q)y(t, q)-\dot{z}(\t, q) ) d\t\\
 &=\int_s^t (\dot{x}(\t, q)(y(t, q)-y(\t, q))+ \dot{x}(\t, q)y(\t, q) -\dot{z}(\t, q) ) d\t\\
  &=\int_s^t (\dot{x}(\t, q)(y(t, q)-y(\t, q))+ H(\t, \vf^{s,\t}(q)) ) d\t
%\\ =y(t, q)\int_s^t\dot{x}(\t, q)d\t - \int_s^t\dot{x}(\t, q)y(\t, q)d\t + \int_s^tH(\t, q, \vf^\t(q))d\t.
\end{align*}
Differentiating respect to $t$ 
\begin{align*}
 \frac d{dt}(  \F^{s,t}(x,y(t, q), z) )
%&= y(t, q)\dot{x}(t, q) \\&+ \dot{y}(t, q)\int_s^t\dot{x}(\t, q)d\t
  %                          -\dot{x}(t, q)y(t, q)+ H( t,\vf^{s,t}(q) )\\ 
  &= \dot{y}(t, q)(x(t, q)-x) + H( t,\vf^{s,t}(q) )\\
&=\dot{y}(t, q)\p_{y} \F^{s,t}(x,y(t, q), z) + H( t,\vf^{s,t}(q)) .
\end{align*}
On the other hand
\[\frac d{dt}(\F^{s,t}(x,y(t, q), z) )=
 \p_{y} \F^{s,t}(x,y(t, q), z) \dot{y}(t, q)
 +\p_t \F^{s,t}(x,y(t, q), z).
 \]
Comparing these expressions we obtain \eqref{Ft=H}.
Similarly we have 
\[\F^{s,t}(\bx(s,r),Y, \bz(s,r)) =\int_s^t (\dot{\bx}(\t,r)(Y-\by(\t,r))+  H(\t, \vf^{t,\t}(r)) ) d\t \]
Differentiating respect to $s$ 
\[\frac d{ds}(\F^{s,t}(x,y(t, q), z) )= (\by(s, r)-Y)\dot{\bx}(s,r) -H( t,\vf^{t,s}(r) ) \]
 On the other hand
 \begin{multline*}
   \frac d{ds}(\F^{s,t}(\bx(s,r),Y, \bz(s,r))=
 \p_x\F^{s,t}(\bx(s,r),Y, \bz(s,r))\dot{\bx}(t, r)\\
+\p_z\F^{s,t}(\bx(s,r),Y, \bz(s,r))\dot{\bz}(t, r)+\p_s\F^{s,t}(\bx(s,r),Y, \bz(s,r))\\
=(y(s, r)-Y)\dot{\bx}(t, r)-\p_z\F^{s,t}(\bx(s,r),Y, \bz(s,r)) H( t,\vf^{t,s}(r) )\\
+\p_s\F^{s,t}(\bx(s,r),Y, \bz(s,r))
\end{multline*}
Comparing these expressions we obtain \eqref{Fs=-H}.
\end{proof}

\section{Generating families for Legendrian submanifols}
\label{sec:gfls}
Let $H\in C_c^2([0,T]\times\jet{k})$ and $\vf^t$ be the contact isotopy  defined by $H$. 
Following \cite{Bh} we will 
construct a \emph{generating families} for Legendrian submanifols.
 
Let $s=t_0<t_1<\cdots<t_N=t$ be a partition such that
$\abs{t_{i+1}-t_i}<\del_H$ so that
$$
\vf^{s,t}=\vf^{t_{N-1}, t_N}\comp\cdots\comp\vf^{t_0,t_1}.
$$
\begin{prop}[Cf. \cite{W1}, corollary 2.8]
 \label{main:thm:01}
Let $0\le s=t_0<t_1<\cdots<t_N=t\le T$ be a partition such that
$\abs{t_{i+1}-t_i}<\del_H$ and $\F^{t_i,t_{i+1}}:\R^{2k+1}\to \R$
be the generating function of $\vf^{t_i,t_{i+1}}$ given by
Proposition \ref{exgf}. For $v\in C^2(\R^k)$ we have
\begin{enumerate}[(a)]
\item One can define a generating family $S^{s,t}:\R^k\times\R^{2kN}\to\R$  of
  $\vf^{s, t}(j^1v)$  by
\begin{equation}
 \label{gfqi}
  S^{s,t}(x;\xi)=
v(x_0)+\sum_{j=1}^N y_j(x_j-x_{j-1})-\F^{t_{j-1},t_j}( x_{j-1}, y_j, z_{j-1})
\end{equation} 
where $x=x_N$, $\xi= (x_0,\ldots,x_{N-1},y_1,\ldots,y_N),$ and $z_0,\ldots,z_{N-1}$ 
are defined inductively by 
\begin{align}\notag
  z_0&=v(x_0)\\ 
 \label{c-0}
  z_j&=z_{j-1}+( x_j-x_{j-1} )y_j-\F_{j-1}( x_{j-1}, y_j, z_{j-1})\quad 
0< j\le N .
 \end{align}
Notice that $z_j$ depends only on 
$(x_0,..,x_j, y_1,\ldots,y_j)$.

\item One can define a $C^2$ function $S^{s,t}:[s,t]\times\R^k\times\R^{2kN}\to\R$, 
such that  each $S^{s,t}(\t,\cdot)$ is a generating family of 
$\vf_{H}^{s,\t}(j^1v)$, as follows: 
let $\t_j=s+(\t-s)\dfrac{t_j-s}{t-s}$ and
\begin{equation}
 \label{gfqi_t}
  S^{s,t}(\t,x;\xi)
  =v(x_0)+\sum_{j=1}^Ny_j(x_j-x_{j-1})-\F^{\t_{j-1},\t_j}(x_{j-1}, y_j, \bz_{j-1}).
\end{equation}
where $x=x_N$, $\xi= (x_0,\ldots,x_{N-1},y_1,\ldots,y_N),$ and 
$\bz_0,\ldots,\bz_{N-1}$ are defined inductively as before

\item For each critical point $\xi$ of $S^{s,t}(\t, x;\cdot)$ we have
  \begin{equation}\label{eq:action}
S^{s,t}(\t,x,\xi)=\int_{s}^{\t} \Big(\dot{x}(\s,j^1v(x_0))y(\s, j^1v(x_0))- 
H\big(\s,\vf^{s,\s}(j^1v(x_0))\big)\Big) d\s
  \end{equation}
where $\vf^{s,\s}(p)=(x(\s,p),y(\s,p),z(\s,p))$, $x(\t,j^1v(x_0))=x$.
\end{enumerate}
\end{prop}
\begin{proof}
We have that the generating function $\F_i:=\F^{t_i,t_{i+1}}:\R^{2k+1}\to \R$
satisfies $1-\pd{z}\F_i\ne 0$ and that conditions 
\begin{equation}
 \label{eq:fg}
 \begin{cases}
  x_{i+1}-x_i=\pd{y}\F_i( x_i, y_{i+1}, z_i ) \\
  y_{i+1}-y_i=-\pd{x}\F_i( x_i, y_{i+1}, z_i )-y_i\pd{z}\F_i(x_i,y_{i+1},z_i )\\
  z_{i+1}-z_i=( x_{i+1}-x_i )y_{i+1}-
\F_i( x_i, y_{i+1}, z_i )
 \end{cases}
\end{equation}
hold if and only if
$\vfd^{t_i,t_{i+1}}(x_i,y_i,z_i)=(x_{i+1},y_{i+1},z_{i+1} ).$ We have
\begin{equation}
\label{yN} 
\p_{x}S^{s,t}(x;\xi)=y_N.
\end{equation}
Let  $i=0,\ldots,N-1.$  For $i<j-1$ we have
\begin{align*}
  \pd{x_i}z_j&= \pd{x_i}( z_{j-1}+(x_j-x_{j-1})y_j-\F_{j-1}(x_{j-1},y_j,z_{j-1}) )\\
  &= \pd{x_i}z_{j-1}-\pd{z_{j-1}}\F_{j-1}\pd{x_i}z_{j-1}= ( 1-\pd{z_{j-1}}\F_{j-1} )\pd{x_i}z_{j-1}
 \end{align*}
and since $\pd{x_i}z_i=y_i$ we get
$$
\pd{x_i}z_{i+1}=y_i-y_{i+1}-\pd{x_i}\F_i-y_i\pd{z_i}\F_i.
$$
As $S^{s,t}(x,\xi)=z_N,$%=\sum\limits_{j=0}^{N-1} z_{j+1}-z_j
for $0< i <N$ we obtain 
\begin{equation}\begin{split}\label{fg:pcx}
\pd{x_i}S^{s,t}(x,\xi)=&(1-\pd{z_{N-1}}\F_{N-1})\cdots(1-\pd{z_{i+1}}\F_{i+1})\\
&\times( y_i-y_{i+1}-\pd{x_i}\F_i-y_i\pd{z_i}\F_i),
\end{split}
\end{equation}

\begin{equation}\begin{split}
\label{fg:pb}
\p_{x_0}S^{s,t}(x;\xi)=&(1-\pd{z_{N-1}}\F_{N-1})\cdots(1-\pd{z_1}\F_{1})\\
&\times(dv(x_0)-y_1-\p_{x_0}\F_0-\p_{z_0}\F_0dv(x_0))
\end{split}
\end{equation}

For $i<j\leq N$
\begin{align*}
  \pd{y_i}z_j&= \pd{y_i}( z_{j-1}+(x_j-x_{j-1})y_j-\F_{j-1}(x_{j-1},y_j,z_{j-1}) )\\
  &= \pd{y_i}z_{j-1}-\pd{z_{j-1}}\F_{j-1}\pd{y_i}z_{j-1}= ( 1-\pd{z_{j-1}}\F_{j-1} )\pd{y_i}z_{j-1},\\
 \pd{y_i} z_i&= \pd{y_i}( z_{i-1} + (x_i-x_{i-1})y_i - \F_{i-1}(x_{i-1}, y_i, z_{i-1}) )\\
 &= x_i-x_{i-1}-\pd{y_i}\F_{i-1},
\end{align*}
so we get
\begin{equation}\label{fg:pcy}
 \pd{y_i}S^{s,t}(x,\xi)=( 1-\pd{z_{N-1}}\F_{N-1} )\cdots (1-\pd{z_{i}}\F_{i} )
( x_i-x_{i-1}-\pd{y_i}\F_{i-1} ).
\end{equation}

From \eqref{gfqi}, \eqref{c-0}, \eqref{eq:fg} and equations
\eqref{fg:pcx}, \eqref{fg:pb} \eqref{fg:pcy} we have that the system
$\p_{\xi}S(x;\xi)=0$, \eqref{yN} is equivalent to 
\begin{align}\notag
\vf^{s,t_1}(x_0, dv(x_0),v(x_0))&=(x_1,y_1,z_1),  \\ \label{enmedio}
\vf^{t_i,t_{i+1}} (x_i, y_i, z_i)&=( x_{i+1}, y_{i+1}, z_{i+1}),\quad
i=1,\ldots, N-2,\\\label{ultimo}
\vf^{t_{N-1},t} (x_{N-1},y_{N-1},z_{N-1})&=(x,\pd{x}S^{s,t}(x;\xi), S^{s,t}(x; \xi)). 
\end{align}
Letting $q_i=(x_i,y_i,\bar z_i)$ we have from Lemma \ref{gen=act}
\[\F^{\t_i,\t_{i+1}}(x_i,y_{i+1},\bar z_i)=y_i(x_{i+1}-x_i)-
\int_{\t_i}^{\t_{i+1}}\dot{x}(\s,q_i)y(\s, q_i)- H(\s,\vf^{s,\s}(q_i)) ) d\s\]
from which item (c) follows. 
\end{proof}

Defining
\begin{align}\label{eq:Q}
Q(\xi)&=-y_Nx_{N-1}+\sum_{i=1}^{N-1}y_i(x_i-x_{i-1}),
\\\label{eq:W}
W^{s,t}(\t,x, \xi)&=v(x_0)+x\cdot y_N-\sum_{j=1}^N \F^{t_{j-1},t_j}( x_{j-1},y_j,\bz_{j-1} ),
\end{align}
we see that for $v\in C^{2,Lip}(\R^k)$, $S^{s,t}(\t,x,\xi)$ is a $\gfqi$

\section{Generalized generating families}

We consider the Cauchy problem \eqref{HJ} with 
$H\in C_c^2([0,T]\times J^1\R^k)$ and $v\in C^{Lip}.$

\begin{prop}[Cf. \cite{W1}, prop. 2.18]
\label{genleg}
Suppose that in the Cauchy problem \eqref{HJ} $v$ is locally Lipschitz
and let $\p v=\set{(x,y,v(x)): y \in \p v(x)}.$
The generating family $S^{s,t}$ given by \eqref{gfqi_t}  generated 
$L^\t=\vf_{H}^{s,\t}(\p v)$ in the sense that
\begin{equation}
 \label{leg}
 L^\t=\set{( x, \p_{x}S^{s,\t}(\t,x,\xi), S^{s,t}(\tau,x, \xi)  ): 0\in \p_{\xi}S^{s,t}(\t,x,\xi)},
\end{equation}
where $\p$ denotes Clarke's generalized derivative and $$\p v=\set{\left( x,y,z \right)\mid y\in \p v(x), z=v(x)}.$$
\end{prop}

\begin{proof}
The condition $0\in \p_{\xi}S(x,\xi)$ means that there exists $y_0\in \p v(x_0)$ such that 
\begin{align}\tag{c1}
y_0-y_1&=\p_{x}\F^{s,t_1}(x_0,y_1,v(x_0) )+y_0\p_{z}\F^{s,t_1}(x_0,y_1,v(x_0) )\\\tag{c2}
y_i-y_{i+1}&=\p_{x_i}\F^{t_i,t_{i+1}}(x_i,y_{i+1},z_i)+\p_{z_i}\F^{t_i,t_{i+1}}(x_i,y_{i+1},z_i) y_i,
\quad 0<i<N\\
 x_i-x_{i-1}&=\p_{y_i}\F^{t_{i-1},t_i} (x_{i-1},y_i,z_{i-1}),
\quad 0<i\le N.\tag{c3}
\end{align}
Since $\p_{x}S^{s,t}(x;\xi)=y_N$, we have that $\vf^{s,t_1}(x_0, y_0,v(x_0))=(x_1,y_1,z_1)$,
\eqref{enmedio} and \eqref{ultimo} hold, and using \eqref{gfqi}
give $\vf^{s,t}(x_0,y_0,v(x_0))=(x,\p_{x}S^{s,t}(x;\xi), S^{s,t}(x;\xi)).$
\end{proof}

\begin{prop}[Cf. \cite{W1}, prop. 2.21]
 \label{cut}
Let $H\in C_c^2([0,T]\times \jet{k})$, $v\in C^{Lip}(\R^k).$ 
Write $S^{s,t}:[s,t]\times\R^k\times \R^q\to\R$ given by \eqref{gfqi_t}
as 
$$
S^{s,t}(\t,x,\xi)=W^{s,t}(\t,x,\xi)+Q(\xi),
$$ 
with $Q$, $W^{s,t}$ defined in \eqref{eq:Q}, \eqref{eq:W} 
%\[W^{s,t}(\t,x,\xi)=v(x_0)+x\cdot y_N-\sum_{j=1}^N\F^{\t_{j-1},\t_j}(x_{j-1}, y_j, \bz_{j-1}).\]
For each compact subset $K$ of $\R^k,$ the  family of functions 
$\{W^{s,t}(\t,x,\cdot)\}_{\t\in[s,t], x\in K}$ is uniformly Lipschitz.
Moreover for any $\th\in C_c(\R^q, [0,T])$ identically $1$ in a
neighborhhod of the origin with $\norm{D\th}< 1$, 
there exists a constant $a_K>1$ such that for $\t\in[s,t]$, %$x\in K$
\begin{equation}
 \label{gfqi_k}
(x,\xi)\mapsto S^{s,t}_K(\t,x,\xi)=\th\Big(\frac{\xi}{a_K}\Big)W^{s,t}(x,\xi)+Q(\xi)
 \end{equation}
is a $\gfqi$ for
$$L^\t_K=L^\t\cap\pi^{-1}(K)=\set{(x, \p_{x}S^{s,t}_K(\t,x,\xi),S_K^{s,t}(\t,x,\xi)| 
0\in \p_{\xi}S^{s,t}_K(\t,x,\xi))},$$ 
where $\pi:\jet{k}\to \R^k, (x,y,z)\to x.$
\end{prop}

\begin{proof}
 For a fixed compact $K,$ let 
$c_K=\max\{\Lip{W^{s,t}(\t,x,\cdot)}:\t\in[s,t], x\in K\}$.
Writing $Q(\xi)=\frac 12\inp{B\xi,\xi}$
\[
 \p_{\xi}S^{s,t}_K(x,\xi)\subset\frac1{a_K}D\th\Big(\frac\xi{a_K}\Big)W^{s,t}(\t,x,\xi)
 + \th\Big(\frac\xi{a_K}\Big)\p_{\xi}W^{s,t}(\t,x,\xi)+B\xi .
\]
Defining $b_{K}=\max\{\abs{W(\t,x,0)}:\t\in[s,t], x\in K\}$ we have
$$
\abs{W^{s,t}(\t,x,\xi)} \leq \abs{W^{s,t}(\t,x,0)} + 
\abs{W^{s,t}(\t,x,\xi)-W^{s,t}(\t,x,0)}
\leq b_{K}+c_{K}\norm{\xi}.
$$
Thus, if $a_{K},b_K$ are sufficiently large,
for $\norm{\xi}\geq b_K$  and any $w\in\p_{\xi}W^{s,t}(\t,x,\xi)$ we have
\begin{align*}
   \abs{\frac1{a_K}D\th\Big(\frac{\xi}{a_K} \Big) W^{s,t}(\t,x,\xi)
 + \th \Big(\frac{\xi}{a_K}\Big)w}
 &\leq \dfrac{1}{a_{K}}( b_{K} + c_{K}\norm{\xi}) +c_{K}\\
 &\leq \frac12\norm{B^{-1}}^{-1}\norm{\xi}\\
& < \norm{B\xi}.
 \end{align*}
We can choose $a_K$ sufficiently large so that $\th\Big(\dfrac{\xi}{a_K} \Big)=1$ 
if $\norm{\xi} \leq b_K$. Thus  $S=S_k,$ for $\norm{\xi} \leq b_K,$ 
and $0\notin \p_{\xi}S_k(x,\xi) $ for  $\norm{\xi}\geq b_K$.
Therefore
$$L^\t_K
=\set{(x,\p_{x}S^{s,t}_K(\t,x,\xi),S^{s,t}_K(\t,x,\xi):0\in\p_{\xi}S^{s,t}_K(\t,x,\xi)}.
$$
\end{proof}

\part{Generalized Solutions of the Cauchy  problem} \label{gen-sol}

In this last part, we will prove our main result, Theorem \ref{thm:main}. First, we will define \emph{minimax solutions} and give some of their basic properties. For this goal, we need a \emph{minimax principle,} which basic definitions and results are given in \cite{STR}. Most of analytic results that we have obtain in this section are generalizations for those in \cite{W1}. Next, using those results, we finally give a demostration of Theorem \ref{thm:main}. Constructions obtained in previous section allows us to follow very closely methods in \cite[section 3]{W1} to achieve our goal. At the end of this section, we present an example. Although it is very simple, this example shows that it is posible to use our results in order to study Hamiltonians with non-compact support.

 \section{Minimax  Selector }\label{min-sel}
Let $K\subset \R^k$ be a compact set, $S_K^{s,t}\in C^1([s,t]\times\R^k\times\R^q)$ 
be $\gfqi$ given as in \eqref{gfqi_k} and  $Q(\xi)=\dfrac12\inp{P\xi,\xi}$ 
be the associated quadratic form. As $S_K^{s,t}=Q$ outside a compact set, the critical 
levels of $S_K^{s,t}$ are bounded. 
There is $R(K)<0$ such that for $\ R'< R(K)$, the sub--level set 
 $$(S_K^{s,t})_{\t,x}^{R'}=\set{\xi\in\R^q|S_k^{s,t}(\t,x;\xi)< R'}
$$ 
is identical to the sub--level $Q^{R'}$.
\begin{defn}\label{fG}
Let $j$ be the Morse index of $Q$ and $a>0$ large. We define
$\fG_a$ as the set of continuous maps $\s:B_j\to\R^q$, of the unit ball
$B_j$ of dimension $j$, such that
$$\s(\p B_j)\subset  Q^{-1}( -\infty, -a).$$
%We denote  $\abs{\s}=\s(B_k)$.
\end{defn}
\begin{lem}\label{RHK}
Let $v\in C^{Lip}(\R^k)$ %and $\vf^t$ be the  contact isotopy defined by 
$H\in C^2_c([0,T]\times J^1\R^k)$.
Let $S_K^{s,t}\in C^1([s,t]\times\R^k\times\R^q)$ be as in \eqref{gfqi_k}.
Let $K\subset \R^k$ be compact, $a>-R(K)$. 
The function 
\begin{equation}\label{us}
 (\t,x)\in[s,t]\times\R^k\mapsto R_{H,K}^{s,\t}v(x)
=\inf_{\s\in \fG_a}\max_{e\in B_j}S_K^{s,t}(\t,x,\s(e)), 
\end{equation}
has the following properties
  \begin{enumerate}[(a)]
 \item $R_{H,K}^{s,\t}v(x)$ is a critical value of $\xi \to S_K^{s,t}(\t,x,\xi);$ 
  \item it is a Lipschitz function and therefore
    differentiable  almost everywhere \emph{a.e.};
  \item $j^1R_{H,K}^{s,\t}v$ is an \emph{a.e.} section of the wave front
$$
\set{(\t,x,\p_{x}S_K^{s,t}(\t,x,\xi),S_K^{s,t}(\t,x,\xi)): 0\in\p_{\xi}S_K^{s,t}(\t,x,\xi)}.
$$ 
\end{enumerate}
 \end{lem}
\begin{proof} Since $S_K^{s,t}(\t,x,\xi)=Q(\xi)$ outside a compact set
  and $Q$ is non-degenerate, we have that 
$$\S_K^{s,t}=\set{(\t,x,\xi)|\p_{\xi}S_K^{s,t}(\t,x,\xi)=0}$$ is compact and
$S_K^{s,t}(x,\t,\cdot):\R^q\to\R$ satisfies the Palais-Smale condition.

Let  $\pi: \S_K^{s,t}\to\R\times\R^k$  be the projection $(\t,x,\xi)\mapsto(\t,x)$.
By Sard's Theorem the set of critival values of $\pi$ has null measure. 

  \paragraph{(a)}
Apply the minimax principle 
(theorem \ref{thm:principiominimax}) with the positively-invariant family $\fG_a$ of Definition \ref{fG}. 

 \paragraph{(b)}
Given $\ep>0,$ there exists $\s_0\in \fG_a$ such that
\[ R_{H,K}^{s,\t_0}v(x_0)\geq \max_{e\in B_j}S_K^{s,t}(\t_0,x_0,\s_0(e))-\ep
   \geq S_K^{s,t}(\t_0,x_0,\s_0(e))-\ep,\] 
for any $e\in B_j$. Let $\max\limits_{e\in B_j}S_K^{s,t}(\t_1,x_1,\s_0(e))= 
S_K^{s,t}(\t_1,x_1, \xi_1)$, then
\begin{align*}
  R_{H,K}^{s,\t_1}v(x_1)-R_{H,K}^{s,t}v(\t_0, x_0)&
\leq S_K^{s,t}(\t_1,x_1,\xi_1)-S_K^{s,t}(\t_0,x_0,\xi_1)+\ep \\
&=\th\Big(\frac{\xi_1}{a_K}\Big)
\big( W^{s,t}(\t_1,x_1,\xi_1)- W^{s,t}(\t_0,x_0,\xi_1)\big)+\ep\\
&\leq A_K(|\t_1-\t_0|+\norm{x_1-x_0})+\ep.
\end{align*}
Letting $\ep\to 0$ and exchanging $(\t_0,x_0)$ and $(\t_1,x_1)$, we get
 $$
 \abs{R_{H,K}^{s,\t_1}v(x_1)-R_{H,K}^{s,t}v(\t_0, x_0)}
\leq A_K(|\t_1-\t_0|+\norm{x_1-x_0}).
 $$

\paragraph{(c)} Let $x_0\in\R^k$ be a regular value of 
$\pi:\S_t\to\R^k$. There is a neighborhood $U$ of $x$ and
difeomorphisms $\phi_i:V_i\to U$, $i=1,\ldots,m$  such that
$$
\pi^{-1}(U)=\bigcup_{i=1}^m \phi_i(U).
$$
For each $x\in U$ there is $i=1,\ldots,m$ such that
\begin{equation}            
 \label{u-valorcrit}
R_{H,K}^{s,\t}v(x)=S_K^{s,t}(\t,x;\phi_i(x))                                                          
\end{equation}
Let $x\in U$ be differentiability point of $R_{H,K}^{s,\t}v$. Proving that there is  
$i=1,\ldots,m$ such that 
\begin{equation}
  \label{eq:meta}
dR_{H,K}^{s,\t}v(x)=dS_K^{s,t}(\t,\cdot; \phi_i(\cdot))(x)  
\end{equation}
will finish the proof. Indeed, as $\p_{\xi}S_K^{s,t}(\t,x,\phi_i(x))=0$
we have
\begin{align*}
dS_K^{s,t}(\t,\cdot;\phi_i(\cdot))(x)&
=\p_{x}S_K^{s,t}(\t,x;\phi_i(x))+\p_{\xi}S_K^{s,t}(\t,x;\phi_i(x))d\phi_i(x)\\
 &=\p_{x}S_K^{s,t}(\t,x;\phi_i(x)).
\end{align*}
and so
$$
j^1R_{H,K}^{s,\t}v(x)\in\{(x,\p_{x}S_K^{s,t}(\t,x;\xi),S_K^{s,t}(\t,x;\xi)):
\p_\xi S_K^{s,t}(\t,x;\xi)=0\}.
$$
To prove \eqref{eq:meta} it suffices to show that there is $i$ such 
that for any unit vector $h$  
\begin{equation}
  \label{eq:h-i}
dR_{H,K}^{s,\t}v(x)\cdot h=dS_K^{s,t}(\t,\cdot; \phi_i(\cdot))(x)\cdot h
\end{equation}
and for that it is enough to show that any unit vector $h$  there is $i=1,\ldots,m$
such that \eqref{eq:h-i} holds, because in such a case there is $i=1,\ldots,m$
such that \eqref{eq:h-i} holds for a base of unit vectors.
Now, there is $\ep>0$ such that for any unit vector $h$ and $|s|<\ep$  
$x+sh\in U$ and so there is $i=i(h,s)$ such that 
$R_{H,K}^{s,\t}v(x+sh)=S_K^{s,t}(x+sh; \phi_i(x+sh))$. For $h$ fixed there is $i=i(h)$ for which
and a sequence $s_k$ converging to zero such that  
$$ R_{H,K}^{s,\t}v(x+s_kh)=S_K^{s,t}(\t,x+s_kh; \phi_i(x+s_kh)) $$
which implies \eqref{eq:h-i}.
\end{proof}

\begin{cor}\label{monotone}
  If $v\le w$ then $ R_{H,K}^{s,\t}v\le R_{H,K}^{s,\t}w$
\end{cor}
 \begin{proof}
  This is clear from \eqref{gfqi_t}, \eqref{gfqi_k} and \eqref{us}.
\end{proof}

\begin{prop}[Cf. \cite{W1}, 2.22]
 \label{ks}
Let $K,K'\subset \R^{k}$ be compact. If $x\in K\cap K',$ $\t\in[s,t]$
then $R_{H,K'}^{s,\t}v( x)=R_{H, K}^{s,\t}v(x).$ 
\end{prop}
\begin{proof}
This follows from the fact for $a>-R(K), a'>-R(K')$,  any $\s\in \fG_a$, 
$\s'\in \fG_{a'}$ can be deformed into an $\s''\in \fG_{a''}$, $a''>a,a'$,   
with
 $$
 \max_{e\in B_j} S_{K}^{s,t}(\t,x,\s''(e))=\max_{e\in B_j}S_{K'}^{s,t}(\t,x,\s''(e)),
 $$
by using the gradient flow de $Q,$ suitable truncated. 
\end{proof}

Propositions \ref{cut} and \ref{ks} allow one to define
\begin{equation}\label{Rs}
 R_{H}^{s,\t}v(x)=\inf_{\s\in\fG_a}\max_{e\in B_j}S^{s,t}(\t, x; \s(e)).
\end{equation}

From Lemma \ref{RHK} we obtain
\begin{thm}
 Function $(\t,x)\in[s,t]\times\R^k\mapsto R_H^{s,\t}v(x)$
has the following properties
  \begin{enumerate}[(a)]
 \item $R_H^{s,\t}v(x)$ is a critical value of $\xi \to S^{s,t}(\t,x,\xi);$ 
  \item it is a Lipschitz function and therefore
    differentiable  almost everywhere \emph{a.e.}.
  \item $j^1R_H^{s,\t}v$ is an \emph{a.e.} section of the wave front
$$\set{(\t,x,\p_{x}S^{s,t}(\t,x,\xi),S^{s,t}(\t,x,\xi)):
  0\in\p_{\xi}S^{s,t}(\t,x,\xi)}.$$ 
\end{enumerate}
\end{thm}
 \begin{prop}
 \label{ind}
The definition of $R_{H}^{s,\t}v(x)$ is independent of the partition of 
$[0,T]$  used to define $S$.
 \end{prop}
\begin{proof}
First assume $t-s<\del_{H};$ and let $\t \in (s,t).$ 
Consider the family of partitions
  $\zeta_{\mu}=\set{s\leq s+\mu(\t-s)<t}$, $\mu\in[0,1]$, 
and the corresponding generating families
  \begin{align}\notag
    S^{s,t}_{\mu}(\t,x; x_0,y_1, x_1,y_2)=v(x_0)
    &+y_1(x_1-x_0)-\F^{s, s+\mu(\t-s)}(x_0,y_1,z_0)\\
    &+y_2(x-x_1)-\F^{s+\mu(\t-s),t}(x_1,y_2,z_1)\label{Smu}
   \end{align}
Function $S_{\mu}$ is continuous in $\m$ and the minimax $R_{S_{\mu}}^{s,t}(\t,x)$ 
is a critical value of the map $\eta \mapsto S_{\mu}^{s,t}(\tau,x; \eta)$.
By  \eqref{eq:action} the set of such critical values is independient of $\mu,$
and by Sard's Theorem, it has measure zero. Therefore $R_{S_{\mu}}^{s,t}$ is 
constant for $\mu\in [0,1].$ 
Letting $x'=x_0-x_1$ y $y'=y_2-y_1$, we obtain 
$$S_0^{s,t}\left( \t,x; x_0,y_1,x_1,y_2 \right)=
v(x_0)-\F^{s,t}(x_1,y_2,z_1)+\left( x-x_1 \right)y_2+x'y'.
$$
One gets this \gfqi adding the 
quadratic form $x'y'$ to the \gfqi
$$
S^{s,t}\left( \t,x; x_0,y_1,x_1,y_2 \right)=
v(x_0)-\F^{s,t}(x_1,y_2,z_1)+\left( x-x_1 \right)y_2,
$$ so that
$$
R_{S}^{s,t}v(x)=R_{S_0}^{s,t}v(x)=R_{S_1}^{s,t}v(x).
$$

In general, given two partitions $\zeta',\zeta''$ of $[s,t]$ with 
$\abs{\zeta'},\abs{\zeta'}< \del_{H},$ let
$$
\zeta=\zeta'\cup\zeta''=\set{s=t_0<\cdots<t_{n}=t}, 
$$ 
be the (smallest) common refinement of $\zeta', \zeta''.$ 
If $t_{j}$ does not belong to $\zeta',$ consider the family of partitions
$$
\zeta_{\mu}(j)=\set{t_0<t_{j-1}\leq t_{j-1}+\mu(t_{j}-t_{j-1})<t_{j+1}<\cdots<t_{n}},\,
\mu\in[0,1]
$$
The argument given at the begining shows that the minimax relative to
$\zeta_0(j)$ and $\zeta_1(j)$ coincide. Continuing this process, 
we obtaing that the minimax relative to $\zeta'$ and $\zeta$ coincide, 
and so do the minimax relative to $\zeta''$ and $\zeta$ as well as 
the minimax relative to  $\zeta'$ and $\zeta''$.
\end{proof}
 
\begin{prop}[Cf. \cite{W1}, prop. 2.24] The critical levels
 $$ C(\t,x):
=\set{\eta:0\in\p_{\eta}S^{s,t}(\t,x,\eta), S^{s,t}(\t,x,\eta)=R_{H}^{s,\t}v( x)}
 $$
are compact and the set-valued correspondence $(\t,x)\to C(\t,x)$ is upper 
semicontinuous, i.e. for every convergent sequence 
$(\t_j,x_j,\eta_j)\to (\t,x,\eta)$ 
with $\eta_j\in C(\t_j,x_j),$ one has $\eta\in C(\t,x)$. In other words the graph 
$\set{(\t,x,\eta)| \eta \in C(\t,x)}$ of the correspondence is closed.
\end{prop}

\begin{proof}
Let $(\t_j,x_j,\eta_j)\to (\t,x,\eta)$ with $\eta_j\in C(\t_j,x_j).$  Since
$S^{s,t}$ is  $C^1$ with respect to $x,$ one has 
$\p S^{s,t}=\p_xS^{s,t}\times \p_\eta S^{s,t},$ which
is upper semicontinuous (lemma \cite[A.2]{W1}). It follows that the limit 
$(\p_xS^{s,t}(\t_j,x,\eta),0)$ of the sequence 
$\p_xS^{s,t}(\t_j,x_j,\eta_j),0\in \p S^{s,t}(\t_j,x_j,\eta_j)$ belongs to 
$\p S^{s,t}(\t,x,\eta)$, hence, $0\in\p S^{s,t}(\t,x,\eta).$
As $S^{s,t}$ and $R_{H}^{s,t}$ are continuous, 
$S^{s,t}(\t_j,x_j,\eta_j)\to S^{s,t}(\t,x,\eta)$, 
$R_{H}^{s,\t_j}v(x_j)\to R_{H}^{s,\t}v(x),$ and therefore $\eta\in C(\t,x).$
\end{proof}

\begin{lem}[Cf. \cite{W1}, lemma 2.25]
 \label{minmaxloc}
 Given $\del >0,$ there exists $\ep>0$ such that
\begin{equation}
\label{minimax:local}
R_{H}^{s,\t}v(x)=\inf_{\sig\in\Sig_{\ep}}\max\{S^{s,t}(\t,x,\s(e)):\s(e)\in C_{\del}(x)\}
\end{equation}
where $$\Sig_\ep=\set{\s\in\fG_a:\max\limits_{e\in B_j}S^{s,t}(\t,x,\s(e))\leq R_H^{s,\t}v(x)+\ep}$$ and 
$C_\del(x)=B_\del(C(\t,x))$ denotes the $\del-$ 
neighborhood of the critical set $C(\t,x).$
\end{lem}

\begin{proof}
 We apply to $S^{s,t}_{\t,x}(\cdot)=S^{s,t}(\t,x, \cdot)$  the following result:
 \begin{lem}[Deformation Lemma]\cite[Theorem 3.4]{STR}
 Suppose $f$ satisfies the Palais-Smale condition. If $c\in\R$ is a
 critical value of $f$ and $N$ any neighbourhood of $K_c:=$ Crit$(f)\cap f^{-1}(c)$, 
then there exist $\ep > 0$ and a bounded smooth
 vector field $V$  equal to $0$ off $f^{c+2\ep}\setminus f^{c-2\ep}$,
 whose flow $\vf^t_V$ satisfies $\vf^t_V(f^{c+2\ep}\setminus N)\subset f^{c-2\ep}$.
 \end{lem}
For $\del>0,$ and $c=R_H^{s,\t}v(x),$ there exist $\ep >0$ and $V,$ a smooth vector
field vanishing outside $(S^{s,t}_{\t,x})^{c+2\ep}\setminus(S^{s,t}_{\t,x})^{c-2\ep}$
such that 
$$
\vf^1_{V}((S^{s,t}_{\t,x})^{c+\ep}\setminus C_{\del}(x))
\subset(S^{s,t}_{\t,x})^{c-\ep}_x.
$$
For $\s\in\fG_a$ we have $\s(B_j) \cap C_{\del(x)}\ne\emptyset$, because otherwise  
$$
\max_{e\in B_j}S^{s,t}(\t,x; \vf^1_{V}(\s(e)))\leq R_{H}^{s,\t}v(x)-\ep
$$
which contradicts the definition of the minimax.

For any $r<c$, the complement of $(S^{s,t}_{\t,x})^{r}$ is a neighborhood of 
$C(\t,x)$. By the same argument one has that  
$\s(B_j) \cap C_{\del(x)}\setminus(S^{s,t}_{\t,x})^{r}\ne\emptyset$.
Therefore, for any $r<c$ and $\sig\in\Sig_\ep$ one has
\[r\le \max\{S^{s,t}(\t,x,\s(e)):\s(e)\in C_{\del}(x)\}\le 
\max_{e\in B_j}S^{s,t}(\t,x,\s(e))\]
wich implies \eqref{minimax:local}.

\end{proof}

\begin{prop}[Cf. \cite{W1}, prop. 2.27]
 \label{w1:prop:2.27}
 The generalized gradient of $R_{H}^{s,\t}v$ satisfies
 \begin{equation}
  \label{w1:2.9}
 \partial R_{H}^{s,\t}v(x)\subset\opname{co}\set{\p_xS^{s,t}(\t,x,\eta):\eta 
\in C(\t,x)},
  \end{equation}
where $\opname{co} $ denotes the convex envelope. 
\end{prop}
\begin{proof}
First we consider a point $\bx$ where $R_{H}^{s,\t}v$ is diferentiable
and prove that
 \begin{equation}
  \label{w1:2.10}
  dR_{H}^{s,\t}v(\bx)\subset \co{\partial_{x}S^{s,t}(\bx, \eta)| \eta \in C(\t,\bx)}.
 \end{equation}
Take $\del,\ep>0$ for $\bx$ as in  Lemma \ref{minmaxloc}. 
Consider $K=\overline{B_1(\bx)}$ and $S_{K}^{s,t}$ as in \eqref{gfqi_k}. 
Choose $B=B_{\r}(\bx)$ with $\r\in(0,1)$ sufficiently small such that for  
$x\in B$
\begin{equation}
\label{2.27:01}
 \abs{S_{K}^{s,t}(\t,x,\cdot)-S_{K}^{s,t}(\t, \bx, \cdot)}_{C^0} < \frac\ep 4.
\end{equation}

\renewcommand{\sl}{\sig_{\lam}}

Let $y\in \R^d, {\lam < 0}$ such that $x_{\lam}=\bx + \lam y \in B$, and
$\lam^2<\frac{\ep}{4}.$ By definition of $R_{H}^{s,t}v$, for each 
$x_{\lam},$ there exists $\sl\in\fG_a$ such that 
\begin{equation}
 \label{2.27:02}
 \max_{e\in B_j} S^{s,t}(\t, \xl, \sl(e))\leq R_{H}^{s,\t}v(\xl) +\lam^2,
\end{equation}
then,
\begin{align*}
\max_{e\in B_j}S^{s,t}(\t,\bx,\sl(e))&\leq \max_{e\in B_j}S^{s,t}(\t, \xl, \sl(e)) +\frac\ep 4 \\
 & \leq R_{H}^{s,\t}v(\xl) +  \frac\ep 2  
 \leq R_{H}^{s,\t}v(\bx) + \frac{3\ep}4
\end{align*}
On the other hand, there exists $\eta_{\lam}\in \sl(B_j) \cap C_{\del}(\bx)$ such that
\begin{equation}
 R_{H}^{s,\t}v(\bx)\leq \max\{S^{s,t}(\t, \bx, \sl (e)):\sl (e)\in C_{\del}(\bx)\} =S^{s,t}(\t, \bx, \eta_{\lam}),
\end{equation}
that implies 
$$
\left( R_{H}^{s,\t}v(\xl)+\lam^2 \right)- R_{H}^{s,\t}v(\bx) \geq S^{s,t}(\t, \xl, \eta_{\lam})-S^{s,t}(\t, \bx, \eta_{\lam}),
$$
since  $\lam < 0,$
\begin{equation}
\label{2.27:05}
\begin{split}
 \frac1{\lam}\left( R_{H}^{s,\t}v(\xl)-R_{H}^{s,\t}v(\bx) \right) 
 & \leq \frac1{\lam}\left( S^{s,t}(\t, \xl, 
\eta_{\lam})-S^{s,t}(\t, \bx,\eta_{\lam}) \right)-\lam \\
&\in\inp{\partial_{x}S^{s,t}(\t,x'_{\lam}, \eta_{\lam}),y}-\lam,
\end{split}
\end{equation}
where the last belonging follows from the Mean Value Theorem \ref{mvt},
for some $x'_{\lam} $ in the line segment between $\bx$ and $\xl$ 

Take $\limsup$ in \eqref{2.27:05} and let $\lam \to 0,$ we get for
all $y\in \R^d:$
\begin{equation}
  \label{eq:2}
\inp{dR_{H}^{s,\t}v(\bx), y} \leq \max_{\eta \in C(\bx)}\inp{\partial_{x}S^{s,t}(\bx, \eta), y}.  
\end{equation}

Considering the convex function 
$f(y)=\max\limits_{\eta \in C(\bx)}\inp{\partial_{x}S^{s,t}(\bx,\eta), y},$ 
inequality \eqref{eq:2} implies 
$$
dR_{H}^{s,\t}v(\bx)\in\partial f(0)=\co{\partial_{x}S^{s,t}(\bx, \eta): \eta \in C(\t,\bx)},
$$

  In the general case 
 \begin{align*}
  \partial R_{H}^{s,\t}v( x)&= \co{\lim_{x' \to x}dR_{H}^{s,\t}v( x')}\\
  & \subset \co{ \co{\lim_{x'\to x}\set{\partial_{x}S^{s,t}(\t,x',\eta'): 
\eta'\in C(\t,x')} }}\\
  &\subset \co{\partial_{x}S^{s,t}(\t,x, \eta): \eta \in C(\t,x)}
 \end{align*}
by the upper semicontinuity of $(\t,x)\to C(\t,x)$ and the continuity of $\p_xS.$
\end{proof}

\section{Viscosity solutions and iterated minimax}
We recall the definition of \emph{viscosity solution}

\begin{defn}
 Let $V\subset \R^{k}$ be open
 \begin{enumerate}[(a)]
  \item A function  $u\in C([0,T]\times V)$ is called a
    \emph{viscosity subsolution} (respectively \emph{supersolution} of 
    \begin{equation}
      \label{eq:HJ}
      \p_tu+H(t,x,\p_x u,u)=0,
\end{equation}
   if for any $\f\in C^1(V\times[0,T])$ and any
    $(t_0,x_0)\in  [0,T]\times V$ at which $u-\f$ has a maximum (respectively
     minimum)  one has
$$
 \pd{t}\f(t_0,x_0)+H(t_0,x_0,\pd{x}\f(t_0,x_0), u(t_0,x_0)) \leq 0
\hbox{ (respectively } \geq 0).
$$

\item The function  $u$ is a \emph{viscosity solution} if it is both
a viscosity subsolution and  a supersolution.  
 \end{enumerate}
\end{defn}
\begin{thm}[\cite{CL86}]
  If $v:\R^{k}\to\R$ is uniformly continuous and $H\in C_c^2( [0,T]\times J^1\R^{k} )$,
 then there exists a unique uniformly continuous viscosity solution of
 the Cauchy problem \eqref{HJ}.
  \end{thm}

\begin{prop}[Cf. \cite{W1}, prop. 3.14]\label{estimates}
 Suppose that $$H\in C_c^2( [0,T]\times J^1\R^{k} ),$$ then the
\emph{minimax operator} $R_{H}^{s,\t}:C^{Lip}(\R^{k})\to C^{Lip}(\R^{k})$ 
satisfies 
\begin{enumerate}[(i)]
\item For $v\in C^{Lip}(\R^{k}),$ 
 \begin{equation}
  \label{w1:prop:3.14.1}
  \norm{\p(R_{H}^{s,t}v)}\leq (\norm{\p v}+\abs{t-s}\norm{\p_xH})
e^{\abs{t-s}\norm{\p_{z}H}}
 \end{equation}
\item There is a constant $C(H)>0$ such that for any $v\in C^{Lip}(\R^{k}),$ 
\begin{equation}
 \label{w1:prop:3.14.2}
\norm{R_{H}^{s,t}v -R_{H}^{s,\t}v} \leq \abs{t-\t}C(H)\norm{H}.
\end{equation}
\item 
If $v^0, v^1\in C^{Lip}(\R^{k})$,  $K\subset\R^{k}$ compact, there exists a
bounded $\til{K}\subset\R^{k},$ depending on $K$ and $\norm{\p v^{i}},$ 
such that for all $0\leq s < t \leq T$: 
\begin{equation}
 \label{w1:3.16}
\norm{R_{H}^{s,t}v^0-R_{H}^{s,t}v^1}_K \leq \norm{v^0-v^1}_{\til{K}}.
\end{equation}
\end{enumerate}
\end{prop}

\begin{proof} 
First we assume that $\abs{t-s}<\del_H$ so that
 \[S(\t,x; x_s, y)=  v(x_s) + (x-x_s)y-\F^{s,\t}(x_s,y,v(x_s))\]
 is a $\gfqi$ for $\vf_{H}^{s,\t}(\p v)$.

\paragraph{(i)}
For $(x_s,y_t)\in C(t,x)$, there is $y_s\in \p v(x_s)$ such that
\[\vf_{H}^{s,t}(x_s,y_s,v(x_s)) = (x,y_t, S(t,x; x_s, y_t))
=(x,y_t,R_{H}^{s,t}v(x)).\] 
As $\p_xS(t,x;x_s, y_t)=y_t$,  by \eqref{w1:2.9} on has 
\[\p_x R_{H}^{s,t}v(x)\subset \co{y_t:y_s\in \p
   v(x_s),\vf_{H}^{s,t}(x_s,y_s,v(x_s))
=(x,y_t, S(t,x; x_s, y_t))}\]
Let $\gam(\t)=(x(\t), y(\t), z(\t))=\vf_{H}^{s,\t}(x_s,y_s,v(x_s))$,
$y_s\in \p v(x_s)$, then
\[
y_t-y_s=\int_s^t\dot{y}(\t) d\t 
 =\int_s^t( -\p_xH(\gam(\t)) -y(\t)\p_zH(\gam(\t)) ) d\t,
\]
$$
\abs{y_t}\le \abs{y_s}+\int_s^t \abs{\p_xH(\gam(\t))}d\t 
+\int_s^t\abs{y(\t)}\abs{\p_zH(\gam(\t))}  d\t.
$$
Hence, by Gr\"onwall's inequality
\begin{align*}
 \abs{y_t} &\leq \Big(\abs{y_s}+\int_s^t\abs{\p_xH(\gam(\t))} d\t\Big)\exp\int_s^t\abs{\p_zH(\gam(\t))} d\t\\%\label{yt:2}
 &\leq (\norm{\p v}+\abs{t-s}\norm{\p_x H})e^{\abs{t-s}\norm{\p_z H}}
\end{align*}
\paragraph{(ii)}
For $(x_s,y_\t)\in C(\t,x)$, there is $y_s\in \p v(x_s)$ such that
$$\vf_{H}^{s,\t}(x_s,y_s,v(x_s)) = (x,y_\t, S(\t,x; x_s, y_\t)).$$
By \eqref{Ft=H}  we have 
$$\p_{\t}S(\t,x; x_s,y_\t)=-H(\t, x,y_\t, S(\t,x; x_s, y_\t)).$$
Hence 
\begin{align*}
\p_{\t} R_{H}^{s,t}v(x)\subset&\hbox{ co }\{-H(\t,x,y_\t, S(\t,x;x_s,y_\t)):
y_s\in \p v(x_s), \\
&\vf_{H}^{s,t}(x_s,y_s,v(x_s))=(x,y_\t, S(t,x; x_s,y_\t))\}
\end{align*}
By the Mean Value Theorem \ref{mvt},
\[\label{Rt-Rs}
 \abs{R_{H}^{s,\t}v(x)-R_{H}^{s,t}v(x)}
 \leq\abs{\t-t}\norm{H}. \]
\paragraph{(iii)}
Consider $v^{\lam}=(1-\lam)v^0+\lam v^1$, $\lam \in [0,1]$ and let
$S_{\lam}^{s,t}$ be the corresponding generating family, then 
$\p_{\lam}S_{\lam}^{s,t}(t,x; x_s, y_t)=v^1(x_s)-v^0(x_s).$
For $(x_s,y_t)\in C^\lam(t,x)$, there is $y_s^\lam\in \p v(x_s)$ such that
$$\vf_{H}^{s,t}(x_s,y_s^\lam,v(x_s)) = (x,y_t,S_\lam^{s,t}(t,x;x_s,y_t)).$$
By a similar argument to the proof of \eqref{w1:2.9} we have
\begin{align*}
 \p_{\lam}R_{H}^{s,t}v^{\lam}(x) &
 \subset \co{\p_{\lam}S_{\lam}^{s,t}(t, x;x_s,y_t):(x_s,y_t)\in C^{\lam}(t,x)}  \\
 &\subset \co{v^1(x_s)-v^0(x_s):x_s\in \til{K}}
\end{align*}
where
\begin{align}\label{tilK}
 \til{K}&=\set{x_s\in\R^k:\norm{x_s}\leq \max_{\bx\in K} \norm{\bx}
+ \abs{t-s}\sup_{y\in Y}\norm{\p_yH}},\\\label{tilk2}
Y&=\set{y \in \R^{k}: \norm{y}
\leq (\norm{\p v}+\abs{t-s}\norm{\p_x H})e^{\abs{t-s}\norm{\p_z H}}}.
\end{align}
Thus we obtain
$$
\abs{R_{H}^{s,t}v^0-R_{H}^{s,t}v^1}_{K}\leq \abs{v^0-v^1}_{\til{K}}.
$$
For the general cases fix $N>T/\del_H$ and take the partition
$t_0<\cdots<t_N$, with $t_i=s+i(t-s)/N$.
\paragraph{(i)}
We have the generating family  $S^{s,t}(x;\xi)$ given in  
\eqref{gfqi} where $$\xi= (x_0,\ldots,x_{N-1},y_1,\ldots,y_N).$$

For $\xi\in C(t,x)$ there exists $y_0\in\p v(x_0)$
such that equations (c1)-(c3) are satisfied. Then 
we have that $\vf^{s,t_i}(x_0, y_0,v(x_0))=(x_i,y_i,z_i)$, $i<N$,
$y_N=\p_xS^{s,t}(x;\xi))$,
$\vf^{s,t}(x_0,y_0,v(x_0))=(x,y_N, S^{s,t}(x;\xi)).$
 By Proposition \ref{w1:prop:2.27}
\[ \p_x R_{H}^{s,t}v(x)\subset \co{y_N:y_0\in \p v(x_0), 
   \vf_{H}^{s,t}(x_0,y_0,v(x_0))=(x,y_N, S^{s,t}(x;\xi))}. \]
Writing $\gam(\t)=\vf_{H}^{s,\t}(x_0,y_0,v(x_0))$ we have 
\[
y_{i+1}=y_i+\int_{t_{i}}^{t_{i+1}}( -\p_xH(\gam(\t)) -y(\t)\p_zH(\gam(\t)) ) d\t. 
\]
By Gr\"onwall's inequality we have as before 
 $$
    \abs{y_{i+1}}\leq \left( \abs{y_i} + \int_{t_i}^{t_{i+1}}\abs{\p_{x}H} \right) 
    \exp\left( \int_{t_i}^{t_{i+1}}\abs{\p_{x}H} \right) 
$$
which imply by induction that 
$$
\abs{y_i}\leq \left( \abs{y_0}+\int_{t_0}^{t_{i}}\abs{\p_{x}H}\right)
\exp\left( \int_{t_0}^{t_{i}}\abs{\p_{z}H} \right).  
$$
Indeed, the inductive step is given by the inequalities
\begin{align*}
 \abs{y_{i+1}}  \leq&\left(  \left( \abs{y_0} + \int_{t_0}^{t_i}\abs{\p_{x}H} \right) 
    \exp\left( \int_{t_0}^{t_i}\abs{\p_{z}H} \right) + \int_{t_i}^{t_{i+1}}\abs{\p_{x}H} \right) \\
    &\times\exp\left( \int_{t_i}^{t_{i+1}}\abs{\p_{z}H} \right) \\
 =&\left( \abs{y_0} + \int_{t_0}^{t_i}\abs{\p_{x}H} \right) 
   \exp\left( \int_{t_0}^{t_{i+1}}\abs{\p_{z}H} \right) \\
 &+\left(\int_{t_i}^{t_{i+1}}\abs{\p_{x}H}\right)\exp\left( \int_{t_i}^{t_{i+1}}\abs{\p_{z}H} \right)\\
 \leq& \left( \abs{y_{t_0}}+\int_{t_0}^{t_{i+1}}\abs{\p_{x}H} \right)
\exp\left( \int_{t_0}^{t_{i+1}}\abs{\p_{z}H} \right) .
\end{align*}
Therefore we have 
$$
\abs{y_N}\leq \left(\norm{\p v}+\abs{t-s}\norm{\p_x H}\right)e^{\abs{t-s}\norm{\p_z H}}. 
$$
\paragraph{(ii)}
Set $D=\sup\set{\norm{g^{\t,\t'}}:\abs{t-\t}<\del_H}$ where $\vf_H^{{\t,\t'}*}\a=g^{\t,\t'}\a$.
We have the generating family  $S^{s,t}(\t,x;\xi)$ given in
\eqref{gfqi_t} with $\bz_0=v(x_0),\bz_1,\ldots,\bz_{N-1}$ 
defined inductively as in \eqref{c-0}. Thus
\begin{align*}
  \frac{d\bz_j}{d\t}=&-
(\p_{\t_{j-1}}\F^{\t_{j-1},\t_j}(x_{j-1},y_j,\bz_{j-1})+
\p_{\t_j}\F^{\t_{j-1},\t_j}(x_{j-1},y_j,\bz_{j-1})) \frac{t_j-s}{t-s}\\
&-\p_z\F^{\t_{j-1},\t_j}(x_{j-1},y_j,\bz_{j-1}) \frac{d\bz_{j-1}}{d\t}
\end{align*}
 Using \eqref{Ft=H} and \eqref{Fs=-H} one proves by
induction that for $j=1,\ldots,N$ one has
\begin{align*}
\frac{d\bz_j}{d\t}&=\sum_{k=1}^{j-1}\prod_{i=k}^{j-1}(1-\p_z\F^{\t_i,\t_{i+1}}(x_i,y_{i+1},\bz_i))
(H(x_{k},y_{k+1},\bz_{k})\\&-H(x_{k-1},y_{k},\bz_{k-1}))\frac{t_k-s}{t-s}\\
 & -H(x_{j-1},y_j,\bz_{j-1}) \frac{t_j-s}{t-s}
\end{align*}
For $\xi\in C(\t,x)$ there exists $y_0\in\p v(x_0)$
such that equations (c1)-(c3) are satisfied. Then 
we have that $\vf^{s,t_i}(x_0, y_0,v(x_0))=(x_i,y_i,z_i)$, $i\le N$,
$y_N=\p_xS^{s,t}(\t,x;\xi))$, $\bz_N=S^{s,t}(\t,x;\xi)$.
We recall that 
\[g^{\t_i,\t_{i+1}}(x_i,y_i,\bz_i)=(1-\p_z\F^{\t_i,\t_{i+1}})(x_i,y_{i+1}\bz_i) \]
so that
$\norm{1-\p_z\F^{\t_i,\t_{i+1}}}\le D$ for $i=0,\ldots,N-1$ and then
\[\abs{\p_\t S^{s,t}(\t,x;\xi)}\le\frac{D^N-1}{D-1}2\norm{H}\]
Since
\begin{multline*}
\p_{\t} R_{H}^{s,t}v(x)\subset\hbox{ co }\{\p_{\t}S^{s,t}(\t,x;\xi):
\vf_{H}^{s,t}(x_0,y_0,v(x_0))\\=(x,\p_xS^{s,t}(\t,x;\xi), S^{s,t}(t,x;\xi))\}
\end{multline*}
by the Mean Value Theorem \ref{mvt},
\[ \abs{R_{H}^{s,\t}v(x)-R_{H}^{s,t}v(x)}
 \leq\abs{\t-t}\frac{D^N-1}{D-1}2\norm{H}. \]
The proof of (iii) does not present changes in the general case.
\end{proof}

Given a subdivision $\z=\set{0=t_0< t_1<\cdots <t_{n}=T}$ of $[0,T],$ 
we define its \emph{norm} as $\abs{\z}=\max \abs{t_{i+1}-t_{i}},$
and the step function 
$$
\z(s)=\max\{t_{i}:t_{i}\le s\},\; s\in[0,T] 
$$
\begin{defn}
 The \emph{iterated minimax operator} for the Cauchy problem \eqref{HJ}
(with respect to $\z$) is defined as follows: for $0\leq s' < s \leq T,$
 \begin{equation}
  R_{H,\z}^{s',s}v(x)=R_{H}^{t_j, s} \comp \cdots \comp R_{H}^{s', t_{i+1}}v(x)
 \end{equation}
where $t_j=\z(s), t_{i}=\z(s').$
When $H$ is fixed, we omit the corresponding subscript.
\end{defn}

\begin{lem}[Cf. \cite{W1}, lemma 3.17]
\label{w1:lem:3.17} Suppose that $(\z_{n})_{n}$ is a sequence of
partitions of $[0,T]$ such that $\abs{\z_{n}}\to 0.$
For $v\in C^{Lip}(\R^{k}),$ the sequence of
functions $u_{n}(s,x):=R_{\z_n}^{0,s}v(x)$ is equi-Lipschitz 
and uniformly bounded on $[0,T]\times K$ 
for any compact $K\subset\R^k$.
\end{lem}
\begin{proof}
It follows from \eqref{w1:prop:3.14.1} that
\[\norm{\p R_{\z_n}^{0,s}v}\le (\norm{\p v}+\abs{s}\norm{\p_xH})
e^{\abs{s}\norm{\p_{z}H}},\]
and from \eqref{w1:prop:3.14.2} that
\[\abs{R_{\z_n}^{0,t}v(x) -R_{\z_n}^{0,s}v(x) } \leq \abs{t-s}C(H)\norm{H},\]
so that in particular
\[\norm{R_{\z_n}^{0,s}v}_K\le \norm{v}_K + T C(H)\norm{H}.\]
The Lemma follows from these inequalities.
\end{proof}
\begin{prop}[Cf. \cite{W1}, prop. 3.18]
 \label{w1:prop:3.18}
 For any sequence $( \z_{n} )$ of subdivisions of $[0,T]$ with 
 $\abs{\z_{n}}\to 0,$ and any compact set $K\subset \R^{k},$ the
 sequence  $u_{n}:=R_{\z_{n}}^{0,s}v(x)$ has a subsequence
 converging uniformly on $[0,T] \times K$ to the viscosity solution of
 the Cauchy problem \eqref{HJ}.
\end{prop}

\begin{proof}
 By lemma \ref{w1:lem:3.17}, we can apply Arzela-Ascoli
 Theorem to $(u_{n})\subset C^0([0,T] \times {K})$ to get
 a subsequence $(u_{n_{k}})$ converging uniformly to a function $\br^{0,s}v.$ 
Define $\til{K}$ as in \eqref{tilK}-\eqref{tilk2}.
\begin{claim}
 \label{w1:prop:3.18-1}
 For $ 0\leq s' < s \leq T$ one has 
 \begin{equation}
 \label{w1:3.18}
  \br^{0,s}v(x)=\lim_{n\to \infty} R_{\z_{n_{k}}}^{s',s}\comp\br^{0,s'}v(x)
\end{equation}
 \end{claim}

\begin{proof}[Proof of claim \ref{w1:prop:3.18-1}]

 Applying Arzela-Ascoli Theorem to $(u_{n_{k}})\subset C^0([0,T]\times\til{K}),$ 
we can extract a subsequence converging uniformly in $[0,T]\times\til{K}.$
To easy notation, when $s'=0$ we omit this superscript, and
for the iterated minimax  with respect 
to the partition $\z_{n},$ we use the subscript $n$
instead, and $(s)_{n}$ instead of $\z_{n}(s).$
 
 We first notice that for $0\leq s \leq T$, $x\in\til{K}$: 
 \begin{equation}
  \label{w1:3.19}
  \br^{s}v(x)=\lim_{n \to \infty}R_{n}^{(s)_{n}}v(x)
 \end{equation}
because
\begin{align*}
 \abs{R_{n}^{s}v(x)-R_{n}^{(s)_{n}}v(x)} &
 =\abs{R^{(s)_{n},s}\comp R_{n}^{(s)_{n}}v(x)-R_{n}^{(s)_{n}}v(x)}\\
&\leq \norm{H}(s-(s)_{n})\leq \norm{H}\norm{\z_{n}}.
\end{align*}
Then for any $\ep>0,$ there exists $N$ such that if $i,j>N,$ then
$$
\forall s\in[0,T]: \norm{R_{i}^{(s)_{i}}v-R_j^{(s)_j}v}_{\til{K}} < \ep.
$$
Therefore,
\begin{align*}
 \norm{R_{i}^{(s')_{i},(s)_{i}}\comp R_{i}^{0, (s')_{i}}v - R_{i}^{(s)_{i}}v}_K
 &= \norm{R_{i}^{(s')_{i},(s)_{i}}\comp R_{i}^{(s')_j}v
 -R_{i}^{(s')_{i}, (s)_{i}}\comp R_{i}^{(s')_{i}}v}_K\\
& \leq \norm{R_j^{(s')_j}v - R_{i}^{(s')_{i}}v}_{\til{K}} < \ep
\end{align*}
Letting $j$ go to $\infty,$ we get 
$$
\norm{R_{i}^{(s')_{i},(s)_{i}}\comp \br^{s'}v- R_{i}^{(s)_{i}}v}_K< \ep,
$$
thus, for any $x\in K$
$$
\lim_{i\to \infty}R_{i}^{(s')_{i}, (s)_{i}}\comp \br^{s'}v(x) =
\br^{s}v(x).
$$
Similarly, we conclude that  
$$
\lim_{i\to \infty}R_{i}^{s',s}\comp \br^{s'}v(x)=
\lim_{i\to \infty}R_{i}^{s',s}\comp \br^{s'}v(x).
$$ 
\end{proof}

We now prove that $\br^{t}v(x)$ is a viscosity subsolution of \eqref{eq:HJ}.
Let $\psi$ be a $C^2$ function with bounded second 
derivative defined in a neighborhood of $(t,x)\in\R\times K$, 
such that for $s$ is  close enough to $t,$  $\psi(s,y)=\psi_{s}(y)\geq \br^{s}v(y)$,
 with equality at $(t,x).$

Suppose that $\t\leq t$ is close enough $t,$ so that the
projections of the characteristics originating from 
$$ j^1(\psi_{\t})(x_\t)=(x_{\t}, d\psi_{\t}(x_{\t}), \psi_{\t}(x_{\t})) $$
do not intersect.  Hence, the map $x_{\t}\to x_{t}$ is a difeomorphism.

We conclude that
 \begin{equation}
  \label{w1:3.20}
   \psi_{t}(x)   =\br^{t}v(x)=\lim_{k\to \infty}R_{n_{k}}^{\t, t} \comp \br^{\t}v(x)
\leq \lim_{k\to \infty}R_{n_{k}}^{\t,t}\psi_{\t}(x)=R^{\t,t}\psi_{\t}(x).
 \end{equation}
 The inequality is consequence of Corollary \ref{monotone}.
 
Also, when $\t$ is close enough to $t,$ iterated minimax will be the
minimax ($N=1$) which is a $C^2$ solution of \eqref{HJ}
with initial condition $\psi_{\t},$ and thus
 \begin{equation}
  \label{w1:3.21}
  R^{\t,t}\psi_{\t}(x)=\psi_{\t}(x)-\int_{\t}^{t}H(\th, j^1(\psi_{\th}(x)))d\th.
 \end{equation}

Subtracting \eqref{w1:3.20} from \eqref{w1:3.21}, $ \psi_{t}(x) $ to
the right side, dividing both sides by $t-\t$ and letting  $\t \to t,$ we get
$$
0\leq -\p_{t}\psi_{t}(x)-H(t, j^1(\psi_{t}(x))). 
$$
One proves that $\br^{t}v(x)$ is a viscosity supersolution of
\eqref{eq:HJ} in a similar way. 
\end{proof}

Given  $H\in C_c^2( [0,T] \times J^1(\R^{k}) )$, $v\in C^{Lip}(\R^{k})$,
we say that a function $w: [s,t]\times \R^k\to\R$
is the limit of iterated minimax solutions for \eqref{HJ} on $[s,t],$ if 
for any sequence of subdivisions $\set{\z_{n}}_{n\in \N}$ of $[s,t]$
such that $\abs{\z_{n}} \to 0$ as $n \to \infty,$ the corresponding
sequence of iterated minimax solutions 
$$\set{R_{H,\z_{n}}^{s,\t}v(x)},(\t,x)\in [s,t]\times \R^k$$ 
converges uniformly on compact subsets to
$w$. We denote $w(\tau,x):=\br_{H}^{s,\t}v(x).$

We can now prove our main result

\begin{proof}[Proof of Main Theorem \ref{thm:main}]
Let $K\subset \R^{k}$ and $(\z_{n})$ be any subsequence of subdivisions
of $[0,T]$ such that $\abs{\z_{n}}\to 0.$ 
Denote $u_{n}(t,x)=R_{\z_{n}}^{0,t}v(x)$ and $u(t,x)$ the viscosity
solution of the \eqref{HJ} problem. If $u_{n}$ does not converge
uniformly on $[0,T]\times K,$ there exists a $\ep >0$ and a
subsequence $n_{k}$ such that $\abs{u_{n_{k}}-u}>\ep.$ Note that
$\z_{n_{k}}$ is itself a sequence of subdivions, this contradicts Proposition
\ref{w1:prop:3.18}.
\end{proof}

\section{Example}

Now we consider an example from \cite{JPRZ}. This example shows that we can ask for weaker conditions, for example, $H$ is not bound to have compact support. Also, this example comes from a very different area than mechanics, namely control theory. Futher directions of this work could be investigating on which conditions we can resemble this result and some applications to other areas like control

Consider $H(x,y,z)=z+h(y)$ with $h$ of compact support. The caracteristics equations are
$$\begin{cases}
   \dot{x}=dh(y)\\
   \dot{y}=-y\\
   \dot{z}=ydh(y)-z-h(y).
  \end{cases}
$$
which can be integrated to obtain the  flow
$$
\vf^t(x_0,y_0,z_0)=(x_0+\int_0^tdh(y_0e^{-s})ds,y_0e^{-t},-h(y_0e^{-t})+e^{-t}(h(y_0)+z_0).
$$
Since that the map
$(x_0,y_0,z_0)\mapsto(x_0,y_0e^{-t},z_0)$ is invertible, 
we can 
%use \eqref{def:gf} to 
define a generating function of $\vf^t$
\begin{align*}
\F^t( x_0, y, z_0 ) 
&=y\int_0^t dh(ye^{t-s})ds + h(y) - e^{-t}h(e^t y) + z_0( 1-e^{-t} )\\  
&=\int_0^t e^{-s}h(e^{s}y) ds + z_0( 1-e^{-t} )
\end{align*}
Thus the minimax solution of  
$$
\begin{cases}
 \p_tu(t,x)+u(t,x)+h(\p_xu(t,x))=0\\
 u(0,x)=v(x)
\end{cases}
$$
is given by
$$u(t,x)=\inf\max S_t(x,x_0,y),$$
where the generating function
$$
S_t(x,x_0,y)=( x-x_0)y-\int_0^t e^{-s}h(e^{s}y) ds + e^{-t}v(x_0)
$$
is quadratic at infinity because $h$ has compact support. 
Indeed, $Q(x_{0},y)=-x_{0}y$ so that
$$
S_t(x,x_0,y)-Q(x_{0},y)=xy-\int_0^t e^{-s}h(e^{s}y) ds + e^{-t}v(x_0).
$$
Since $S_{t}(x,x_0,y)$ is $C^{1}$ with respect to $y,$ we have
\begin{align*}
\p_{(x_{0},y)}\left( S_t(x,x_0,y)-Q(x_{0},y) \right)
%&=\p_{x_{0}}\left( S_t(x,x_0,y)-Q(x_{0},y) \right)\times \p_{y}\left( S_t(x,x_0,y)-Q(x_{0},y) \right)\\
&= e^{-t}\p v(x_{0})\times \set{x-\int_{0}^{t}dh(e^{s}y)ds}\\
 &=\set{\left( e^{-t}p, x-\int_{0}^{t}dh(e^{s}y)ds \right): p \in \p v(x_{0})}.
\end{align*}

As $h$ is compactly supported and $v$ is a Lipschitz function, 
$$
\norm{\p_{(x_{0},y)}\left( S_{(t,x)}-Q\right)}=\max_{(x_{0},y)}
\set{\Big\|\left( e^{-t}p, x-\int_{0}^{t}dh(e^{s}y)ds \right)\Big\|: p \in \p v(x_{0})}
$$ 
is bounded, and therefore $S_{t}$ is a \texttt{gfqi.}

Had we assumed instead that $h$ was convex we would still had obtained a minimax
\[ u(t,x)=\inf_{x_0}\max_y \Big((x-x_0)y-\int_0^t e^{-s}h(e^{s}y)ds\Big)+e^{-t}v(x_0) \]
with the $\max_y$ a Legendre transform being achieved when
\[ x-x_0=\int_0^t dh(e^{s}y)ds.\]
Letting $l$ to be the Legendre transform of $h$, it is not hard to prove that
\[u(t,x)=\min_y\int_0^te^{-s}l(dh(e^{s}y))ds+v\left(x-\int_0^tdh(e^{s}y)ds\right),\]
(cf. \cite[(4)]{JPRZ}, Hopf-Lax formula with discount.)
%\appendix 

%\input{./AP-04_symplectic.tex}
%\input{./AP-02-truco_chekanov.tex}
%\input{./AP-01-grandientes_generalizados.tex}
% 
% 
% 
\part{Appendix}
% % 
% % The appendix fragment is used only once. Subsequent appendices can be created
% % using the part Section/Body Tag.
\section{Contact topology}
\label{subsec:char_dir}

As noticed in \cite{book:5877}, due to a more complex geometry than ordinary differential equations (\texttt{ode's}), there not exists a unified theory for partial differential equations (\texttt{pde's}). 

For an \texttt{ode}, we can always consider locally integrable vector fields, that is, there are always integral curves for them.
However, even an hyperplane field on $\R^{3}$ is not always integrable. 

For example, let us examine the hyperplane field given by equation 
$$\a:=dz-ydx=0.$$ Now, consider $v\in \evat{\ker \a}{(x,y,z)}, \ v=\left( v_{x},v_{y},v_{z} \right).$ 
Thus, $\evat{\a}{(x,y,z)}$ is a linear form induced with associated matrix $R_{\a}=[-y,0,1].$ 
The last equation states that $v_{z}=yv_{x}.$ 
% Therefore
% $$v=\left( v_{x},v_{y},yv_{x} \right)
% =v_{x}\left( 1,0,y \right) +v_{y}(0,1,0).$$ 

% \begin{figure}[t]
%  \centering
%  \includegraphics[height=5cm,keepaspectratio=true]{./Standard_contact_structure.png}
%  % Standard_contact_structure.png: 640x290 pixel, 72dpi, 22.58x10.23 cm, bb=0 0 640 290
%  \caption{Standard Contact Structure on $R^3$}
%  \label{fig:01}
% \end{figure}

This hyperplane field is not integrable (that is, there is no submanifold $N$ such that at each $p\in N,$ $T_{p}N$ is on the hyperplane field at $p$) because of  \emph{Frobenius integrability condition} $\left( \a\wed d\a=0 \right)$ does not hold at $\a$ inasmuch as
$$
\a\wed d\a=(dz-ydx)\wed(-dy\wed dx)=dx\wed dy \wed dz \neq 0,
$$
that is, $a \wed d\a$ is a \emph{volume form} on $\R^{3}.$ 

In this part, we shall give definitions and results from contact geometry and theory of first-order p.d.e's in order to make clear statements of subsequent results in this thesis.

%\subsection{Legendrian Submanifolds}
%\section{First-Order Partial Differential Equations}

Any $\texttt{pde}$ of first order can be written as
\begin{equation}
\tag{F=0}
 \label{fopde}
 F(x_{1},...,x_{n}, \partial_{x_{1}}u(x),...,\partial_{x_{n}}u(x), u(x))=0,
\end{equation}
where $x=(x_{1},...,x_{n}).$ So any $\texttt{pde}$ of first order can be regarded as a hypersurface $\pdehs$  in  $J^{1}(\R^{n})\iso T^{*}\R^{n}\times \R \iso \R^{2n+1},$ and a solution of this $\texttt{pde}$ as a function $u:M\to \R$ whose $1-$graph
$$
j^{1}u=\set{j^{1}_{x}u| x\in M}.
$$
(where $j^{1}_{x}u:=(x,\partial_{x}u(x),u(x))$) lies in $\pdehs$.

Instead of $\R^{n},$ one can consider an $n-$ dimensional manifold $M,$ and in that case  we ontain the space $J^{1}M\iso T^{*}M\times \R.$

Indeed, a $1-$graph is a section over $M.$ Therefore $J^{1}M$ is also a vector bundle over $M$ with projection 
$$
\pi_{M}:J^{1}(M)\to M, (x,y,z)\mapsto x.
$$

$J^{1}(M)$ is not just a differentiable manifold, but has an analogous strucutre to the symplectic structure of the cotangent bundle $T^{*}M.$ Now, we will explain some details about this contact struture. At the end of the section, we will define a very important concept, namely \emph{contactomorphism}, that is, transformations preserving this structure. For example, every solution $\pdehs$ of \eqref{fopde} is a legendrian submanifold of $J^{1}M$ and the flow $\vf^{t}$ generated by the characteristic equations consists of contactomorphisms. 

\begin{defn}
 A \emph{contact structure} on $\conmani$ is a field of hyperplanes $\xi=\ker \a \subset T\conmani$  where $\a$ is a $1-$form satisfying $\a \wed \left( d\a \right)^{n}\neq 0.$ We say that $\a$ is a \emph{contact form,} and the pair $(\conmani, \xi)$ is a \emph{contact manifold.} 
\end{defn}

\begin{exmp}[Standar structure]
For $J^{1}\R^{n}\iso T^{*}\R^{n}\times \R,$ we can choose $\a=dz-ydx$ in local coordinates $(x,y,z).$ 
It is straightforward to verify that $J^{1}(M)$ is a contact manifold and such $\a$ is a contact form: Consider the $1-$form $\a$ in local coordinates, i.e.,
\begin{equation}
 \label{alpha}
 \a=dz-ydx=dz-\sum_{i=1}^{n}y_{i}dx_{i}.
\end{equation}

Then
$$
d\a = -dy\wed dx = -\sum dy_{i}\wed dx_{i},
$$
and therefore
$$
(d\a)^{n}=(-1)^{n} n! \bigwedge_{i=1}^{n} \left( dy_{i}\wed dx_{i} \right).
$$

Thus
$$
\a\wed (d\a)^{n} = (-1)^{n}n! dz \bigwedge_{i=1}^{n} \left( dy_{i}\wed dx_{i} \right) \neq 0.
$$
%because it defines a \emph{volume form} on $J^{1}(M).$

\end{exmp}

\begin{rem}
 Indeed, along this work we will just consider the contact structure in $J^{1}\R^{n}$ induced by $\a=dz-ydx.$ In the next paragraphs we will explain why it is enough to choose this contact structure.
\end{rem}

 Notice also that if $\a$ is a contact form, then $\a\wed (d\a)^n$ is a \emph{volume form} and therefore, $\conmani$ is orientable. However, some authors define $\xi=\ker \a$ just locally and in that case $\conmani$ is not necessary an orientable manifold anymore. 
 
 If there were another form such that $\xi=\ker \b,$ then we would have $\b=\lam \a,$ for some function $\lam \in C^{\infty}(\conmani, \R \bs 
\set{0}),$ because $$\codim \xi=\dim \im \a= 1.$$ 

Since
$$
\b \wed (d\b)^{n}=(\lam \a)\wed (d(\lam\a))^{n}=\lam \a \wed (\lam d\a + d\lam \wed \a)^{n} =\lam^{n+1}(\a \wed (d\a)^{n}) \neq 0
$$
\emph{the definition of a contact structure $\xi$ does not depend on a particular choice of $\a$.}

\begin{claim} %\cite[Prop. 2]{book:5877}
\cite[Rem. 1.3.7]{book:72479}
 \label{contacto:alt}
An equivalent definition of $\xi\subset T\conmani$ as a contact structure over a manifold of dimension $(2n+1)$ is as follows: For any local $1-$ form $\a$ with $\xi = \ker \a,$ $(d\a)^{n}|_{\xi}$ is non-degenerate, that is, 
$(\xi_{p}, d\a|_{\xi_{p}})$ is a symplectic space for every $p\in \conmani.$
\end{claim}

% % \begin{proof} Choose a local trivialization
% % $$
% % \set{e_{1},f_{1},...,e_{n},f_{n},r} 
% % $$
% % of $T\conmani=\xi \oplus \xi',$ such that
% % $\xi=\ker \a= \inp{e_{1},f_{1},...,e_{n},f_{n}}$ y $\xi'=\inp{r}.$ Thus
% % $$
% % (\a \wed (d\a)^{n})(e_{1},f_{1},...,e_{n},f_{n},r)=\a(r)\cdot (d\a)^{n}(e_{1},f_{1},...,e_{n},f_{n}),
% % $$
% % with $\a(r)\neq 0$ and therefore
% % $$
% % \a \wed (d\a)^{n}\neq 0 \iff (d\a)^{n} \neq 0 \iff d\a\at{H}\text{ is non-degenerate.}
% % $$
% % \end{proof}

As in symplectic topology, the following result states that there are not local invariants in contact topology:
\begin{thm}[Darboux]\cite[Th. 2.5.1]{book:72479} \label{darboux} Let $(\conmani,\xi)$ be a contact structure and $p\in \conmani.$ Thus there exists a coordinate system $$(U,x_{1},y_{1},...,x_{n},y_{n},z)$$ centered around $p$ such that in $U$ and a diffeomorphism $f:U\to f(U)\subset J^{1}\R^{n}$ such that $Df\left( \xi \right)=\ker \a$ by defining
\begin{equation}
 \label{aest}
 \a=dz-ydx.
\end{equation}

\end{thm}

%\subsection{Legendrian Submanifolds}

A submanifold is called integral if the tangent plane at each point is a subspace of the contact plane. For example, a $1-$graph is always an integral submanifold of $J^{1}(M),$ of the same dimension as $M.$ For $\s(x)=\left( x, \s_{y}(x), \s_{z}(x) \right):$
\begin{align*}
 \s=j^{1}\vf &\imply  \s_z(x)=\vf(x), \,  \s_y(x)dx=d\vf(x)\\
&\imply \s^{*}\a=0.
\end{align*}
Since $\a\left( D\s(x)(v) \right)=0,$ we have
\begin{equation}
 \label{Dsx}
  \im D\s(x) \subset \ker \evat{\a}{\s(x)},
\end{equation} that is, every tangent plane of every $1-$graph at a given point lies in the same hyperplane.

Indeed, there are not integral submanifolds of higher dimension.
\begin{defn}
 Let $(\conmani, \xi)$ be a contact manifold. A submanifold $L$ of $(\conmani, \xi)$ is called \emph{isotropic} if 
$T_{p}L\subset \xi_{p}$ for every $p\in L.$
\end{defn}

\begin{prop}\cite[Prop. 1.5.12]{book:72479}
 Let $(\conmani,\xi)$ be a contact manifold of dimension $2n+1$ and $L\subset (\conmani,\xi)$ a isotropic submanifold. Then $\dim L \leq n.$
\end{prop}

% % \begin{proof}
% %  Let $i:L\inc \conmani$ be an injection and $\a$ a contact form defining $\xi,$ at least locally. Thus isotropicity condition can be restated as $i^{*}\a=0.$ So that $i^{*}d\a=0.$ In particular, 
% % $T_{p}L\subset \xi_{p}$ is an isotropic subspace of the symplectic $2n-$dimensional vector space  $(\xi_{p}, d\a|_{\xi_{p}}).$  By well-known results from linear symplectic algebra \cite[see][remark 1.3.6]{book:72479} %\todo[inline]{Write down these well-know results from linear symplectic algebra}
% % $ 
% % \dim T_{p}L \leq (\dim \xi_{p})/2 =n.
% % $
% % \end{proof}

\begin{defn}
 An isotropic submanifold $L\subset \left( \conmani, \xi \right), \; \dim\conmani=2n+1$ of maximal dimension $n$ is called Legendrian. \end{defn}

 \begin{claim} Let $(\conmani, \xi)=(J^{1}\R^{n},\ker \a).$ It follows form \eqref{Dsx} and the fact that $$\dim\left( j^{1}f \right)=\dim(M)$$ that $1-$graphs are Legendrian submanifolds of $J^{1}(M).$
\end{claim}

Another two important concepts in the study of geometry of our Cauchy problem are \emph{contatomorphism} and \emph{contact isotopy}. The first one is a diffeomorphism preserving contact structures and the second one is a family of contactomorphisms varying in the time. Our most important example of a contact isotopy is the flow generated by \eqref{bh:1.3a}-\eqref{bh:1.3c}. At the end of the section, we will give a result about how to recovery a classical solution for \eqref{HJ} from this flow. From here and because Theorem \ref{darboux}, we will just consider contact manifolds $\left( J^{1}\R^{n}, \ker \a \right)$ with $\left( x,y,z \right)$ local coordinates for $J^{1}\R^{n}\iso T^{*}\R^{n}\times \R$ and $\a=dz-ydx.$

Let $\xi=\ker \a$ a contact structure for contact manifold $J^{1}M.$ A diffeomorphism $ \psi:J^{1}M \to J^{1}M $ is called a \emph{contactomorphism} if $ \psi $ preserves the oriented hyperplane field $ \xi. $ This is equivalent to the condition
\[ \psi^{*}\a = e^{h}\a \]
for some function $ h:J^{1}M \to \R. $ A contact isotopy is a smooth family $ \psi_{t}:J^{1}M\to J^{1}M $ of contactomorphism such that
\[ \psi^{*}_{t}\a = e^{h_{t}}\a. \]

% Suppose these contactomorphisms are generated by smooth vector field $ X:J^{1}M \to TJ^{1}M $ via
% \[ \dfrac{d}{dt}\psi_{t}=X\comp \psi_{t}, \, \psi_{0}=\id. \]
% 
% Let $\L$ be the \emph{Lie derivative}. Recall that if the flow $\psi_{t}$ is generated by vector field $X,$ then for a given $k-$form
% %\cite[see][section 2.2(f)]{book:8247}. Because of identity
% $$
% \L_{X}\om=\evat{\dfrac{d}{dt}}{t=0}\psi^{*}_{t}\om,
% $$
%  \cite[see][prop. 2.13]{book:8247}. The following calculation
% \begin{align*}
%  \psi_{t_{0}}^{*}\L_{X}\a
%  &= \psi_{t_{0}}^{*}\left( \evat{\dfrac{d}{dt}}{t=0}\psi^{*}_{t}\a \right)\\
%  &=\evat{\dfrac{d}{dt}}{t=t_{0}}\psi_{t}^{*}\a\\
% &=\left( \evat{\dfrac{d}{dt}}{t=t_{0}}h_{t} \right)\left( e^{h_{t_{0}}}\a \right)\\
% &=\left( \evat{\dfrac{d}{dt}}{t=t_{0}}h_{t}  \right)\psi_{t_{0}}^{*}\a\\
% &=\psi_{t_{0}}^{*}\left( g_{t_{0}}\a \right).
% \end{align*}
% with $g_{t_{0}}=\left( \left( \evat{d/dt}{t=t_{0}} \right)h_{t} \right)\comp \psi_{t_{0}}^{-1}.$
%  
% In particular, when $t_{0}=0$ we have $\L_{X}\a=g\a$ with $g:=\evat{g_{t}}{t=0}$ and this identity shows that
% \[\L_{X}\a = g\a.\]
% 
% Conversely, if $X$ satisfies this condition then the diffeomorphism $\psi_{t}:J^{1}M \to J^{1}M$ generated by $X$ as above determine a contact isotopy with \[h_{t}=\int_{0}^{t}g_{s}\comp \psi_{s}ds.\] 

A vector field $X:J^{1}M \to TJ^{1}M$ which satisfies $\L_{X}\a=g\a$ for some function $g: J^{1}M \to \R$ is called a \emph{contact vector field}.

For every contact form $\a$, we define the \emph{Reeb vector field} $R_{\a}$ as the unique one for which the following equations hold
\begin{equation}
 \label{reeb}
 \tag{Reeb}
 \begin{cases}d\a(R_{a},v )=0, & v\in T_{p}J^{1}M \\ \a(R_{\a})=1. \end{cases}
\end{equation}

Such field exists because for every $p\in J^{1}M,$ $\ker d\a|_{T_{p}J^{1}M}$ is one-dimensional, and $R_{\a}$ is defines except for a rescaling and the second condition allows us choose it uniquely. The importance of this vector field is given by the following result:
\begin{lem}\cite[lemma 3.49]{book:7935}
 \label{bh:lem:1.1}
 Let $(\conmani,\xi=\ker \a)$ be a contact structure with Reeb field $R_{\a}.$ Thus:
 \begin{enumerate}[(i)]
  \item $X:\conmani \to T\conmani$ is a contact vector field if and only if there exist a function $H:\conmani \to \R$ such that
\begin{equation}
\label{ist:3.10}
\begin{cases}
 i(X)\a=-H, \\ \ i(X)d\a= dH-\left( i(R_{\a})dH \right)\a.
 \end{cases}
\end{equation}
\item For every function $H:\conmani\to \R,$ there exists an unique contact vector field $X_{H}:\conmani \to T\conmani$ which satisfies
\eqref{ist:3.10}
 \end{enumerate}
\end{lem}

% % \begin{proof}
% %  If \eqref{ist:3.10} holds, then $\L_{X}\a=g\a,$ where $g=-i(R_{\a})dH.$ Conversly, suppose that
% % $\L_{X}\a=g\a$ and define $H=-i(X)\a.$ Thus
% % $$
% % i(X)d\a=\L_{X}\a-di(X)\a=dH+g\a.
% % $$
% % Evaluating this 1-form at $R_{\a},$ we obtain $i(R_{\a})dH+g=0,$ and in this way, we have proof the first statement.
% % 
% % Now, consider a given function $H:\conmani \to \R$. Then, there exists a unique vector field $Z:\conmani \to T\conmani,$ such that $Z \in \xi = \ker \a$ and
% % \begin{equation}
% %  \label{zinkera}
% %  i(Z)d\a \at{\xi}=dH\at{\xi}.
% % \end{equation}
% % 
% % 
% % By Darboux' Theorem \ref{darboux}, at least locally we can choose a local coordinates such that $$\conmani \iso J^{1}\R^{n}, \; \dim \conmani =2n+1$$ and $\a=dz-ydx.$ So w.o.l.g. let's take
% % $$
% % Z=\left( Z_{x}, Z_{y}, Z_{z} \right) \in \ker \a,
% % $$
% % such that $Z_{z}-yZ_{x}=0.$ Hence
% % $$
% % Z=\left( Z_{x}, Z_{y}, yZ_{x} \right).
% % $$
% % 
% % Let $W=(W_{x}, W_{y}, W_{z})$ be another arbitrary vector on $\ker \a,$ such that using \eqref{zinkera}, we have
% % $$
% % -dy\wed dx (Z, W)=dH(W), 
% % $$
% % and therefore
% % $$
% % W_{x}(-Z_{y})+W_{y}(Z_{x})=
% %  W_{x}\left( \partial_{x}H + y\partial_{z}H \right)
% %  +W_{y}(\partial_{y}H),
% % $$
% % from which we deduce
% % $$
% % \begin{cases}
% %  Z_{x}=\partial_{y}H\\
% %  Z_{y}=-\partial_{x}H-y\partial_{z}H.
% % \end{cases}
% % $$
% % 
% % The vector field which we are looking for is defined by
% % $$X_{H}=Z-HR_{\a}.$$
% % \end{proof}

\begin{exmp}
 \label{exmp:reeb}
 The Reeb vector field for the standard contact $1-$form $ \a=dz-ydx $ on $J^{1}\R^{n}$  is $\partial_{z}.$ Let's verify this claim: If $R=\left( R_{x},R_{y},R_{z} \right),$ then
 $$
 0=i(R)d\a=-R_{y}dx+R_{x}dy,
 $$
 implies $R_{x}=R_{y}=0.$ Finally
 $$
 1=\a(R)=R_{z}-yR_{x}=R_{z}.
 $$
 We conclude that $ R=(0,0,1)=\partial_{z}.$ So
\[X_{H}=Z-HR_{\a}=\left( Z_{x},Z_{y},Z_{z} \right)-H(0,0,1)
=\left( \p_{y}H, -\p_{x}H-y\p_{z}H, y\partial_{y}H-H \right),\]
and therefore, the flow generated by $X_{H}$ is given precisely by \eqref{bh:1.3a}, \eqref{bh:1.3b} and \eqref{bh:1.3c}.
\end{exmp}

For a given function $H\in C_{c}^{2}\left( [0,T]\times J^{1}\R^{k}
\right) ,$ the following statement relates solutions
$S:[0, T]\times \R^{k} \to \R$ for 
\emph{Cauchy problem of First Order} \eqref{HJ} (being $S_{0}:\R^{k} \to \R$ the initial condition) and solutions for characteristic equations \eqref{bh:1.3a}-\eqref{bh:1.3c} starting at the $1-$graph of $S_{0}.$

\begin{prop}\cite[Prop. 1.3]{Bh}
 \label{bh:prop:1.3}
 Suppose that $S\in C^2([0,T] \times \R^{k})$ is a solution for the given Cauchy problem \eqref{HJ}. Then, if a solution for the equation
\begin{equation}
     \label{xdot}
\dot{x}=\pd{y}H\left( t,x, \pd{x}S(t,x), S(t,x) \right),
\end{equation}
on $[0,T]$ is given, the curve defined by
\begin{equation}
\label{hjcc}
 \left(x(t),y(t),z(t)\right)=\left( x(t), \pd{x}S(t,x(t)), S(t,x(t)) \right).
\end{equation}
is a solution for characteristic equations \eqref{bh:1.3a}-\eqref{bh:1.3c}.

Conversely,  $S_{0}\in C^{1}(M)$, $T$ is small enough and  $$(X_{(t,x)},Y_{(t,x)},Z_{(t,x)}):[0,T]\to J^{1}\R^{k}$$ is the unique solution for characteristic equations \eqref{bh:1.3a}-\eqref{bh:1.3c} with boundary conditions
$$
\begin{cases}
X_{(t,x)}(t)=x, \\ 
Y_{(t,x)}(0)=\pd{x}S_0X_{(t,x)}(0)), \\
Z_{(t,x)}(0)=S_0(X_{(t,x)}(0)),
\end{cases}
$$ 
then 
\begin{equation}
\label{aux:1.15}
S(t,x)=Z_{(t,x)}(t)
   \end{equation} 
defines the solution for the Cauchy problem with initial condition $S(0,x)=S_{0}(x).$
\end{prop}

If $(\vf^{t})_{t\in[0,T]}$ is the isotopy related to $H$ 
that is,
\begin{equation}
 \label{gt}
 \frac{d}{dt}\vf^{t}=X_{H}(\vf^{t});
\end{equation} 
proposition \ref{bh:prop:1.3} says that if $S(x,t)=S_t(x)$ is a solution for \eqref{HJ}, then
\begin{equation}
 \label{geosol}
 j^{1}S_{t}=\vf^{t}\left( j^{1}S_{0} \right),
\end{equation} 
and conversely, if $t$ is small enough, then
the solution for the Cauchy problem \eqref{HJ} with initial condition $S_0$ is defined by \eqref{geosol}.

However, there could happen some issues; first of all, contact transformation $\vf^{t}$ 
could no longer be well-defined at each point of $j^{1}S_{0}.$ In addition, although $\vf_t( j^{1}S_{0})$ was well-defined, it could happen that it is not a section of $\pi:J^{1}\R^{k}\to \R^{k}, \ \left( x,y,z \right)\to x$ anymore.
%\end{rem}

%\input{./jdcastillo_thesis_02.tex}

\section{Clarke Calculus}
\label{calculoclarke}

In this part, we will show some results from \emph{Clarke calculus} for \emph{generalizaed gradients,} used in parts above.

%\subsection{Definitions}

We say that $\R^{m}\to\R^{}$ \emph{Lipschitz of rank $K$} near a given point $x\in \R^{m},$ if for some point $\ep >0,$ 
we have
$$
\abs{f(y)-f(z)}\leq K\norm{y-z},
$$
for all $y,z\in B(x,\ep).$

The \emph{generalized directional derivative} of $f$ at $x$ in the direction $v$ is defined as
$$
f^{o}(x;v):=\limsup_{y\to x, t\downarrow 0}\dfrac{f(y+tv)-f(y)}{t},
$$
for $y\in \R^{m}$ and $t>0.$

A function $g$ is called \emph{positively homogeneous} if $g(\lam v)=\lam g(v)$ for $\lam\geq 0,$ and subadditive if for every $v,w:$
$$
g(v+w)\leq g(v)+g(w).
$$

\begin{defn}
 A function $F:X\to \R^{}$ is called \emph{upper semicontinuous } if: 
 $$
 v_{i} \in X \to v \in X \implies \limsup_{i\to \infty}F(v_{i})\leq F(v).
 $$
\end{defn}

\begin{prop}\cite[Prop. 10.2]{CL1}
 \label{cl:prop:10.2}
 Let $f$ be a Lipschitz function of rank $K$ near $x.$ Hence:
 \begin{enumerate}[(a)]
  \item $v\mapsto f^{o}(x;v)$ is finite, positively homogeneous and subadditive function on $\R^{m},$ and for all $v\in \R^{m}:$
$$
\abs{f^{o}(x;v)}\leq K\norm{v}.
$$
\item For all $v\in \R^{m},$ the map $(u, w)\mapsto f^{o}(u;w)$ is upper semicontinuous at $(x;v),$ 
and $w \mapsto f^{o}(x;w)$ is Lipschitz of rank $K$ on $\R^{m}.$
%\item Para toda $v\in \R^{m}:$ $f^{o}(x;-v)=(-f)^{o}(x;v).$
 \end{enumerate}
 
The following results allows us to prove the above one:
 
 \begin{thm} \cite[Theorem 4.25]{CL1}
 \label{cl:thm:4.25}
  Let $g\in \R^{m}\to \R\cup \set{\infty}$ a lower semicontinuous, subadditive and positively homogeneous funtion such that $g(0)=0.$ Then there exists a unique convex subset $\Sig\subset \R^{m}$ such that
  $g$ is the support function $H_{\sig}$ of $\Sig,$ that is, for every $x\in \R^{m}:$
  $$
  g(x)=H_{\sig}(x):=\sup_{\sig \in \Sig}\inp{\sig, x}.
  $$
  
  The set $\Sig$ is characterized by
  $$
  \Sig=\set{\eta \in \R^{m}| g(v)\geq \inp{\eta, v}, v \in \R^{m}},
  $$
  and it is a compact one if and only if $g$ is bounded on the unitary disc.
 \end{thm}

%  \begin{proof}[Proof, prop. \ref{cl:prop:10.2}]
%   For the sake of brevity, we will not give a demonstration of the first statement.
%   
%   For the second one, whereas $v\to f^{o}(x;v)$ is a lower semicontinuous, subadditive and positively homogeneous function, bounded on the unitary disc, because of Theorem \ref{cl:thm:4.25}, this is the support function of a convex subset on $\R^{m}.$ 
%  \end{proof}
 \end{prop}

\begin{defn}
 The \emph{generalized gradient} of a function $f$ on $x,$ denoted by $\pd{}f(x)$ in the unique compact convex non-empty subset of
$\R^{m}$ whose support function is $f^{o}(x; \cdot).$ Therefore:
$$\begin{cases}
   \zeta \in \partial f(x) \iff \forall v\in \R^{m}: f^{o}(x;v)\geq \zeta \cdot v\\
   \forall v\in \R^{m}: f^{o}(x;v)= \max\set{\zeta \cdot v| \zeta \in \partial f(x)}
  \end{cases}
$$
\end{defn}

As corollary, the \emph{generalized gradient} is well-defined.

%\subsection{Properties}

\begin{prop}
 Let $f$ be a Lipschitz function of rank $K$ near $x.$ Then $\partial f(x)\subset B(0,K).$
\end{prop}

\begin{proof}
For any $\eta \in \partial f(x),$  by proposition \ref{cl:prop:10.2} we have that for all $v\in \R^{m}$
 $$
 \inp{\eta, v} \leq K \norm{v}.
 $$
\end{proof}

If $f$ is Lipschitz near $x,$ and differentiable at $x,$ then $f'(x)\subset \partial f(x),$ and the following result allows to generalize the concept of a critical point:
\begin{lem}[Fermat Rule]
 Let $f:\R^{m} \to \R$ be locally Lipschitz:
 \begin{enumerate}[(a)]
  \item If $f$ has a local maximum or minimum at $x,$ then $0\in \partial f(x).$
  \item If $0\notin \partial f(x),$ there exists a decreasing direction $v$ for $f$ at $x,$ that is,
  $$
  \limsup_{t \downarrow 0}\dfrac{f(x+tv)-f(x)}{t} <0.
  $$
 \end{enumerate}
\end{lem}

We can obtain a version of the following classical result for generalized calculus:

\begin{thm}[Mean Value Theorem]\cite[Theorem 10.17]{CL1}
 \label{mvt} 
 Given $x,y \in \R^{m}$ such that $f$ is Lipschitz in a neighborhood of interval
 $$
 [x,y]=\set{\lam x + (1-\lam)y, \lam \in [0,1]}.
 $$
 
 Hence. there exist a point at $(x,y)=\operatorname{Interior}[x,y]$ such that
 $$
 f(y)-f(x)\in \inp{\partial f(z), y-x}.
 $$
\end{thm}

We will use the following special case of the \emph{chain rule}:
\begin{lem}
Suppose that $f$ is Lipschitz in a neighborhood of interval $[x,y].$ Define $x_{t}=x+t(y-x), t\in[0,1]$ and 
 $$
 g:[0,1] \to \R, g(t)=f(x_{t}).
 $$
 
 Then $\partial g(t) \subset \inp{\partial f(x_{t}), y-x}.$
\end{lem}

\begin{proof}
 Since both sets in the inclusion are indeed intervals on $\R,$ it is enough to show that for every $v=\pm 1:$
 $$
 \max\set{\partial g(t)v} \leq \max \set{\inp{\partial f(x_{t}), y-x}v}.
 $$
 However, in this case $g^{o}(t;v)=\partial g(t).$ Hence 
 \begin{align*}
  g^{o}(t;v)&= \limsup\set{\frac{g(s+\lam v)-g(s)}{\lam} \at{} 
  s\to t,  \lam \downarrow 0 } \\  
  &=\limsup_{_{s\to t,\lam \downarrow 0}} \set{\frac{f(x+[s+\lam v](y-x)) -f(x+s(y-x))}{\lam} } \\  
  &\leq\limsup\set{\frac{f(z+\lam v (y-x)) -f(z)}{\lam} \at{} 
  z\to x_{t},  \lam \downarrow 0 } \\  
  &=f^{o}(x_{t}; v(y-x))=\max \inp{\partial f(x_{t}), v(y-x)}.
 \end{align*}

\end{proof}

\begin{proof}[Proof of the Mean Value Theorem \ref{mvt}]
Consider the function $:[0,1]\to \R$ defined by
 $$
 \th(t)=f(x_{t})+t\left( f(x)-f(y) \right).
 $$
 Note that $\th(0)=\th(1)=f(x),$ so that there exists $t^{*}\in[0,1]$ at which $\t$ attains an extremal value. In this case, $0\in \partial \th(t^{*}),$ so that 
 $$
 0\in f(x)-f(y) +\inp{\partial f(x_{t^{*}}), y-x},
 $$
and we have proved our Theorem when $z=x_{t^{*}}.$
\end{proof}

Finally. the following result says that the \emph{generalized gradient} is closed:
\begin{thm}\cite[prop. A.2]{W1}
\label{w1:A.2}
Let $f$ be Lipschitz of rank $K$ near $x$ and $x_{i}, v_{i}$ sequences on $\R^{m}$ such that
 $$\begin{cases}
 x_{i} \to x\\ \zeta_{i} \in \partial f(x_{i}).
 \end{cases}$$
If $\zeta_{i} \to \zeta,$ then $\zeta \in \partial f(x).$
\end{thm}

\begin{proof}
 Fix $v \in \R^{m}.$ For every $i,$ we have that $$f^{o}(x_{i};v) \geq \zeta_{i}\cdot v.$$ The sequence 
$\zeta_{i}\cdot v$ is bounded on $\R$ and it contains term which are arbitrarly closed to $\zeta\cdot v.$ 
Extract a subsequence of $\zeta_{i}$ (which we denote in the same way as the original one) such that
$\zeta_{i} \cdot v \to \zeta \cdot v.$ Taking the limit in the above inequality and whereas $f^{o}$ 
is upper semicontinuous en $x,$ we conclude that
$$
f^{o}(x;v)\geq \zeta \cdot v.
$$

Since $v$ is arbitrary, it follows that $\zeta \in \partial f(x).$
\end{proof}

\section{The Minimax Principle}

\label{ap:principiominimax}

%\subsection{Palais-Smale Condition}

\begin{defn}
 \begin{enumerate}
  \item A sequence $(u_{m})$ in a manifold $M$ is called of \emph{Palais-Smale} for $E\in C^{1}(M)$ if $\abs{E(u_{m})} \leq c$ uniformly in $m$ and 
  $$
  \lim_{m\to \infty} \norm{DE(u_{m})} = 0.
  $$
  \item We say that $E$ satisfies \emph{Palais-Smale compactness condition  $(PS)$} if every P-S sequence for $E$ has a strongly convergent subsequence.
 \end{enumerate}

\end{defn}

Palais-Smale condition allows us to distinguish certain family of neighborhood of critical point of a given functional $E;$ and hence, it will useful to characterize regular values of $E.$

For $\b \in \R,$ $\del >0,$ $\r>0$ define
\begin{align*}
 E_{\b}&=\set{u \in V: E(u)<\b}\\
 K_{\b}&=\set{u\in V: E(u)=\b, DE(u)=0}\\
 N_{\b, \del}&=\set{u \in V: \abs{E(u)-\b}<\del, \norm{DE(u)}<\del}\\
 U_{\b, \r}&=\bigcup_{u\in K_{\b}}
 \set{v \in V: \norm{v-u}< \r}.
\end{align*}

\begin{prop}\cite[lemma 2.3]{STR}
 \label{str:lem:2.3}
Suppose that $E$ satisfies $(P.S).$ Then for every $\beta \in \R,$ it follows that:
\begin{enumerate}
 \item $K_{\beta}$ is a compact subset;
 \item both $\set{U_{\b, \r}}_{\r>0}$ and $\set{N_{\b, \del}}_{\del >0}$ are fundamental systems of neighborhoods for $K_{\b}.$
\end{enumerate}

 \end{prop}

 \begin{rem}
  In particular, if $K_{\b}=\emptyset$ for some $\b \in \R$ there exists $\del>0$ such that $N_{\b, \del}=\emptyset;$ that is, the differential $DE(u)$ is uniformly bounded in norm for all $u\in V,$ far enough from origin but with $E(u)$ close to $\b.$
 \end{rem}

 %\subsection{A very general deformation lemma}

 Denote by
 $$
 \til{V}=\set{u \in V: DE(u)\neq 0}
 $$ the set of regular points of $E.$ Instead of a gradient, which requires the existence of well-defined inner product, we will use the following:
 \begin{defn}
  \label{str:defn:3.1}
  A \emph{pseudogradient} vector field for $E$ is a locally Lipschitz continuous one $v\in \til{V} \to V$ for which the following conditions hold:
  \begin{enumerate}
   \item $\norm{v(u)}< 2 \min\set{\norm{DE(u)},1};$
   \item for all $u\in \til{V}:$
   $$
   \inp{v(u), DE(u)} > \min\set{\norm{DE(u)},1}\norm{DE(u)}.
   $$
  \end{enumerate}

 \end{defn}

 \begin{lem}\cite[lemma 3.2]{STR}
  \label{str:lem:3.2}
  Any functional $E\in C^{1}(V)$ admits a pseudogradient vector field $v:\til{V}\to V.$
 \end{lem}

\begin{thm}[Deformation Lemma]\cite[Theorem 3.4]{STR}
Suppose that $E\in C^{1}(V)$ satisfies $(PS).$ Fix $\b \in \R,$ $\bar{\ep}>0$ and $N$ some neighborhood of $K_{\b}.$ Thus there exist $\ep \in (0, \bar{\ep})$ and a uniparametric continuous family of homeomorphisms $$\F(\cdot,t):V\to V, \, 0\leq t < t$$ with the following properties:
\begin{enumerate}
 \item $\F(u,t)=u,$ if $t=0$ or $DE(u)=0$ or $\abs{E(u)-\b}\geq \bar{\ep};$
 \item $E(\F(u,t))$ is non-decreasing  on $t$ for all $u\in V;$
 \item $\F(E_{\b+\ep}\backslash N,1)\subset E_{\b-\ep},$ and $\F(E_{\b+\ep},1)\subset E_{\b-\ep}\cup N.$
\end{enumerate}

Moreover $\F:V \times [0,\infty) \to V$ has semigroup property, that is, 
$$
\forall s,t \geq 0: \, \F(\cdot, t) \circ \F(\cdot,t)=\F(\cdot,s+t).
$$
\end{thm}

Since $\F:V \times [0,\infty)$ is obtained integrating a truncated pseudogradient vector field in suitable manner, $\F$ is called local pseudogradient flow.

%\subsection{Minimax Principle}

\begin{defn}
 Let $\F:M \times [0,\infty) \to M$ be a semiflow in a manifold $M.$ A family $\mF$ of subsets of $M$ is called positively $\F-$invariant if $\F(F,t)\in \mF$ for every $F \in \mF, t \geq 0.$
\end{defn}

%\begin{thm}[Lema de Deformaci\'on]
% %La demostraci\'on completa est\'a en la tesis de Wei
% %teorema B.10
% \label{w1:A.6}
% Dado $c\in \R$ y una vecindad $U$ de $crit(S)\cap f^{-1}(c)$ en $\R^{k}\times\R^{q},$ existe $\ep>0$ tal que $\vf_{1}\left( S^{c+\ep}\bs U \right)\subset S^{c-\ep}.$ 
% %En particular, si $c$ es un valor regular de $f,$ podemos elegir $U=\emptyset.$
%\end{thm}

\begin{thm}[Minimax Principle]\cite[Theorem 4.2]{STR}
\label{thm:principiominimax}
 Suppose that $M$ is a complete Finsler manifold of class $C^{1,1}$ and $E\in C^{1}(M)$ satisfies the $(PS)$ condition. Also suppose that $\mF \subset \mP(M)$ is a collection of subsets invariant with respect to every semiflow $\F:M \times [0, \infty) \to M$ satisfying
 \begin{enumerate}[(a)]
 \item $\F(\cdot, 0)=\id,$
 \item $\F(\cdot,t)$ is a homeomorphism of $M$ for every $t\geq 0,$ and 
 \item  $E(\F(u,t))$ is non-decreasing in $t$ for every $u \in M.$ 
 \end{enumerate}
 
 Therefore, if
 $$
 \b=\inf_{F \in \mF} \sup_{u \in F} E(u)
 $$
is finite, $\b$ is a critical value of $E.$
\end{thm}

\begin{exmp}
 Let $X$ be a topological space and $[X,M]$ the set of free-homotopy classes $[f]$ of continuous maps  $f:X \to M.$ For $[f]\in[X,M]$ define 
 $$
 \mF=\set{g(X) | g\in[f]}.
 $$
 
 Since $[\F \comp f]=[f]$ for any homeomorphism $\F$ de $M$ homotopic to identity, the family $\mF$ is invariant under such maps $\F.$ Hence if
 $$\b=\inf_{F \in \mF} \sup_{u\in F} E(u)$$
 is finite, $\b$ is a critical value.
\end{exmp}

%\subsection{Palais-Smale conditions in the Context of Clarke Calculus}

\begin{defn} A point $\x\in X$ is called \emph{critical} for $f$ if $0\in \p f(\x);$ the value $f(\x)$ is called \emph{critical} for $f.$ Note that the \emph{critical set} $crit(f)$ of $f,$ consisting of every critical point is closed on $X.$

Define  
$$
\lam(\x)=\min_{w\in \p f(\x)}\norm{w}_{X^{*}}.
$$
We say that $f$ satisfies the \emph{Palais-Smale condition} $(PS)$ if every subsequence $(\x_{n})$ such  that $f(\x_{n})$ is bounded  and $\lam(\x_{n})\to 0$ has a convergent subsequence whose limit is a critical point of $f$ and thus there exists $y_{n}\in \p f(\x_{n})$ such that $y_{n}\to 0.$ 
\end{defn}

\begin{prop}[C.f. \cite{W1}, example A.4]
  \label{w1:exmp:A.4}
 $(PS)$ condition holds where $\Lip{f-Q}\leq \infty$ for some non-degenerate quadratic form  $Q:X\to\R.$ In this case $crit(f)$ is \emph{compact.}
\end{prop}

\begin{proof}
If we define $\psi=f-Q,$ each subset $\p f(\x)=\p \psi(\x)+ dQ(\x)$ consists of vectors whose norm is at least $\norm{dQ(\x)}-\Lip{\psi},$ y hence $$\lam(x)\geq \norm{dQ(\x)}-\Lip{\psi},$$ so that $\lam(x)\to \infty$ where $\norm{x}\to \infty.$ Therefore, there exists $R>0$ such that each sequence $(x_{n})$ with $\lim \lam(x)=0$ satisface $\norm{x_{n}} \leq R$ for $n$ large enough, and this follows both $(PS)$ condition and compactness of $crit(f).$
\end{proof}

 \bibliography{testBib}{}
 \bibliographystyle{plain}
\end{document}